\documentclass[11pt, english, draft]{article}
\usepackage{amssymb,amsmath,amsthm,tikz}
\usepackage{graphicx,color}
\usetikzlibrary{shapes.geometric}

\usetikzlibrary{graphs}
\usetikzlibrary{graphs.standard}

\newtheorem{thm}{Theorem}[section]

\newtheorem{cor}[thm]{Corollary}

\newtheorem{lem}[thm]{Lemma}

\newtheorem*{coj*}{Conjecture}
\title{\bf Concerning Some Properties of  Signed Graphs Associated With Specific Graphs}

\author{{\bf Y. Bagheri, A. R. Moghaddamfar} and {\bf F. Ramezani$^{1}$}\\[0.5cm]
\date{}
{\em In memory of Professor Michael Neumann.}}

\begin{document}
\newcommand{\f}{\frac}
\newcommand{\sta}{\stackrel}
\maketitle
\footnotetext[1]{\tt Corresponding author, email:
ramezani@kntu.ac.ir.}

\begin{abstract}
\noindent Two signed graphs are called switching isomorphic if one of them is isomorphic to a switching equivalent of the other. To determine the number of switching non-isomorphic signed graphs on a specific graph,
we will establish a method based on the action of its automorphism group.
As an application and computational results, we classify all the switching non-isomorphic signed graphs arising from the complete graph $K_5$ and the generalized Petersen graph ${\rm GP}(7,2)$.
Moreover, some results on the automorphism groups of the target signed graphs are obtained.
\end{abstract}

\section{Introduction}
All graphs considered in this paper are, simple and undirected. Let $\Gamma$ be a graph of order $n$ with vertex set $V_\Gamma=\{v_1, v_2, \ldots, v_n\}$ and edge set $E_\Gamma$.  For vertices $v_i, v_j\in V_\Gamma$, we will write $v_i\sim v_j$ if  $v_iv_j\in E_\Gamma$. The {\em adjacency matrix} ${\mathbf A}={\mathbf A}(\Gamma)$ of a graph $\Gamma$ of order $n$ is the symmetric $(0,1)$-matrix $[a_{ij}]_{n\times n}$ such that $a_{ij}=1$ if $v_i\sim v_j$, and $0$, otherwise.

A \textit{signed graph}  (or {\em sigraph}) is an ordered pair $\Sigma=(\Gamma, \sigma)$ consists of a simple graph $\Gamma=(V_\Gamma, E_\Gamma)$, referred to as its {\em underlying graph} of $\Sigma$, and  a {\em sign function} $\sigma: E_\Gamma \rightarrow \{+1, -1\}$. The concept of signed graphs is given by Harary \cite{Harary}.
Recently there are some papers which have studied signed graphs from different perspectives (see for instance \cite{Et-F, GHS, Naser, Qi, Soz, Yu, ZAS, ZAS2}).
We denote by ${\cal S}(\Gamma)$ the set of all signed graphs  with underlying graph $\Gamma$.
For a signed graph $\Sigma=(\Gamma, \sigma))$, the vertex set $V_\Sigma$ of $\Sigma$ coincide with the vertex set of its underlying graph, while the edge set $E_\Sigma$ is divided into two disjoint subsets $E_\Sigma^+$ and $E_\Sigma^-$ (defined by $\sigma$) that contain positive and negative edges, respectively. For a signed graph $\Sigma=(\Gamma,\sigma)$, by $\Sigma^+$ and $\Sigma^-$, we mean the following unsigned graphs:
 $$\Sigma^+=(V_\Gamma, E_{\Sigma^+}) \ \ \ \ \mbox{and} \ \ \ \  \Sigma^-=(V_\Gamma, E_{\Sigma^-}),$$
 where
$$E_{\Sigma^+}=\sigma^{-1}(+1) \ \ \ \ \mbox{and} \ \ \ \  E_{\Sigma^-}=\sigma^{-1}(-1),$$

Given a signed graph $\Sigma=(\Gamma, \sigma)$ with vertex set $V_{\Sigma}=V_\Gamma=\{v_1, v_2, \ldots, v_n\}$, the adjacency matrix of $\Sigma$ is defined as a square matrix ${\mathbf A}_\sigma={\mathbf A}(\Sigma)=[a_{ij}^\sigma]_{n\times n}$ with
$$a_{ij}^\sigma=\sigma(v_iv_j)a_{ij},$$
where $a_{ij}$ is the $(i, j)$-entry of the adjacency matrix of the underlying graph $\Gamma$,  or equivalently,
$$a_{ij}^\sigma=\left\{ \begin{array}{rl} 1, & \hbox{if $v_i\sim v_j$ and $\sigma(v_iv_j)=+1$,} \\
-1, & \hbox{if $v_i\sim v_j$ and $\sigma(v_iv_j)=-1$,} \\
0, & \hbox{if $v_i\nsim v_j$.}  \end{array}   \right. $$
A {\em switching function} is a function $\theta: V_\Gamma \rightarrow  \{+1,-1\}$. A signed graph $\Sigma$ is transformed by a switching function $\theta$ to a new signed graph $\Sigma^\theta = (\Gamma, \sigma^\theta)$ such that the underlying graph remains the same and the sign function $\sigma^\theta$ is defined on an edge $v_iv_j\in E_\Gamma$ by $$\sigma^\theta(v_iv_j)=\theta(v_i) \sigma(v_iv_j) \theta(v_j).$$ Two signed graphs $\Sigma_1, \Sigma_2\in {\cal S}(\Gamma)$
 are said to be {\em switching equivalent}, denoted by $\Sigma_1\sim \Sigma_2$, if there exists a switching function $\theta$ such that $\Sigma_2=\Sigma_1^\theta$, or there exists a diagonal matrix $${\mathbf D}^\theta= {\rm diag}(\theta(v_1), \theta(v_2), \ldots, \theta(v_n)),$$ such that $${\mathbf A}(\Sigma_2) = {\mathbf D}^\theta {\mathbf A}(\Sigma_1) {\mathbf D}^\theta,$$
 otherwise
they are called \textit{switching non-equivalent}.
Recall that {\em switching} of a signed graph at a vertex $v$ reverses the sign of each edge incident with $v$.
Hence, alternatively we might say that two signed graphs $\Sigma_1, \Sigma_2\in {\cal S}(\Gamma)$ are
switching equivalent if $\Sigma_2$ is obtained from a series of switchings of $\Sigma_1$.

It is clear that being `switching equivalent' defines an equivalence relation on ${\cal S}(\Gamma)$, and therefore it partitions the set ${\cal S}(\Gamma)$ into equivalence classes.
We denote the equivalence class containing the signed graph $\Sigma\in {\cal S} (\Gamma)$ by $[\Sigma]$. The set of all
equivalence classes will be denoted by $\Omega_s(\Gamma)$.

Two signed graphs $\Sigma_1=(\Gamma_1,\sigma_1)$ and $\Sigma_2=(\Gamma_2,\sigma_2)$
are \textit{isomorphic}, denoted by $\Sigma_1\cong \Sigma_2$, if there is a graph isomorphism from $\Gamma_1$ to $\Gamma_2$, which preserves the signs of
edges. Similarly, the signed graphs $\Sigma_1$ and $\Sigma_2$  are said to be \textit{switching isomorphic} if
$\Sigma_1$ is isomorphic to a signed graph which is switching equivalent to $\Sigma_2$, otherwise we
call them \textit{switching non-isomorphic}.

Given a graph $\Gamma$ with $n$ vertices and $m$ edges, there are $2^m$ ways of constructing a signed graph on $\Gamma$. When $\Gamma$ is connected,  we show that there are $2^{m-n+1}$ mutually switching non-equivalent  signed graphs in ${\cal S}(\Gamma)$. Zaslavsky in \cite{ZAS2} proved that there are only six different signed Petersen graphs, up to switching isomorphism. Motivated by his results, we will investigate  the same problem for other lists of graphs.
We also define an action of the automorphism group ${\rm Aut}(\Gamma)$ on the set $\Omega_s(\Gamma)$. Using some well-known results from Group Theory,  we count the number of switching non isomorphic signed graphs. The main idea is determining orbit size of each class $[\Gamma, \sigma]$. Classifying those signed graphs with a same underlying graph which are mutually switching non-isomorphic, significantly decrease the number of signed graphs for studying their properties. For the notation and definitions on group theoretical aspects, we refer the reader to \cite{Isaacs}.

We now introduce some further notation and definitions on Graph Theory. Two distinct edges $e_i$ and $e_j$ are said to be {\em adjacent} if and only if they have a common end vertex.
Two edges of a graph are {\em disjoint} if they do not share a common vertex.  Two graphs $\Gamma_1=(V_{\Gamma_1}, E_{\Gamma_1})$ and $\Gamma_2=(V_{\Gamma_2}, E_{\Gamma_2})$ are {\em isomorphic}, denoted by $\Gamma_1\cong \Gamma_2$, if there exists a bijection $\varphi: V_{\Gamma_1}\rightarrow V_{\Gamma_2}$ such that, for every pair of vertices $v_i, v_j\in V_{\Gamma_1}$, $v_iv_j\in E_{\Gamma_1}$ if and only if $\varphi(v_i)\varphi(v_j)\in E_{\Gamma_2}$.
An {\em automorphism} of a graph $\Gamma$ is a graph isomorphism between $\Gamma$ and itself.
For a graph $\Gamma$, we denote by ${\rm Aut}(\Gamma)$ the set of all automorphisms of $\Gamma$ which forms a group under the operation of composition, and we call this group the {\em automorphism group} of $\Gamma$.

The sequel of this paper is organized as follows: In Section 2 we provide some preparatory results.
The main results are
presented in Section 3. Finally, we find the number of signed graphs with the underlying graph $K_5$ and the generalized Petersen graph ${\rm GP}(7,2)$,
up to switching isomorphism.


\section{Auxiliary Results}
We start with the following classical result, which is proved in \cite[Proposition 3.1] {Naser}.
We give here an elementary proof of this result to make the paper
self-contained.
\begin{lem} \label{nds} {\rm (\cite[Proposition 3.1] {Naser})}
Let $\Gamma$ be a connected graph with $n$ vertices and $m$ edges. Then, there are $2^{m-n+1}$ mutually switching non-equivalent signed graphs with underlying graph $\Gamma$.
\end{lem}
\begin{proof} We apply double counting to the set ${\cal S}(\Gamma)$  containing all signed graphs on $\Gamma$. On the one hand,
 there are $2^m$ possible sign functions on $E_\Gamma$, and so $|{\cal S}(\Gamma)|=2^m$.

 On the other hand,  let ${\cal S}^\ast(\Gamma)$ denote the set of mutually switching non-equivalent signed graphs. Recall that, there are $2^n$ possible switching functions on $V_\Gamma$.
 Suppose that $\theta$ is a switching function on $V_\Gamma$ and $\Sigma=(\Gamma, \sigma)\in {\cal S}^\ast(\Gamma)$ is a signed graph. We first argue that $\Sigma^\theta=\Sigma^{-\theta}$. Indeed, we have
$$\sigma^{-\theta}(v_iv_j)=(-\theta(v_i)) \sigma(v_iv_j) (-\theta (v_j))=\theta(v_i) \sigma(v_iv_j) \theta(v_j)=\sigma^\theta (v_iv_j),$$  for every $v_iv_j\in E_\Gamma$, as desired.

Next, we claim that if $\theta_1\neq \theta_2$ are two arbitrary switching functions on  $V_\Gamma$
 such that  $\Sigma^{\theta_1} = \Sigma^{\theta_2}$, then $\theta_2=-\theta_1$.
 Indeed, since $\theta_1\neq \theta_2$, there exists $u\in V_\Gamma$  with $\theta_1(u)\neq \theta_2(u)$.
 Let $v$ be an arbitrary vertex in $V_\Gamma \setminus \{u\}$.
 Since $\Gamma$ is connected, there exists a path between $u$ and $v$ in $\Gamma$, say
 $$u=w_1, w_2, w_3, \ldots, w_{n-1}, w_n=v.$$
It now follows from the equality $\Sigma^{\theta_1} = \Sigma^{\theta_2}$ that for
every $i=1, 2, \ldots, n-1$, we have $$\sigma^{\theta_1}(w_iw_{i+1})=\sigma^{\theta_2}(w_iw_{i+1}),$$
or equivalently, $$\theta_1(w_i)\sigma(w_iw_{i+1})\theta_1(w_{i+1})=\theta_2(w_i)\sigma(w_iw_{i+1})\theta_2(w_{i+1}),$$
that is, $$\theta_1(w_i)\theta_1(w_{i+1})=\theta_2(w_i)\theta_2(w_{i+1}).$$
Since $\theta_1(w_1)\neq \theta_2(w_1)$,  one obtains $\theta_1(w_2)\neq \theta_2(w_2)$. Continuing with our argument for $i=2, 3, \ldots, n-1$, we finally conclude that $\theta_1(w_n)\neq \theta_2(w_n)$, that is $\theta_1(v)\neq \theta_2(v)$. Since $v$ was arbitrary, it follows that $\theta_2=-\theta_1$, as claimed.
Therefore, any signed graph $\Sigma\in {\cal S}^\ast(\Gamma)$  admits $2^{n-1}$ switching non-equivalent signed graphs, which implies that  $$|{\cal S}(\Gamma)|=|{\cal S}^\ast(\Gamma)|\times 2^{n-1}.$$

Finally, we conclude that  $2^m=|{\cal S}^\ast(\Gamma)|\times 2^{n-1}$, from which the lemma follows at once. \end{proof}

In what follows we will let $c$ denote the number of connected components of $\Gamma$. As an immediate consequence of Lemma \ref{nds} we have:

\begin{cor} \label{cor-1}
Let $\Gamma$ be a graph with $n$ vertices, $m$ edges and $c$ connected components. Then, there are $2^{m-n+c}$ mutually switching non-equivalent signed graphs with underlying graph $\Gamma$.
\end{cor}
\begin{proof}  This is immediate from Lemma \ref{nds} and the fact that  every graph is disjoint union of its connected components. \end{proof}

We make use of the following lemma, due to T. Zaslavsky see \cite{ZAS}, in our paper.
\begin{lem} {\rm (\cite[T. Zaslavsky]{ZAS})} \label{HA}
Two signed graphs $\Sigma_1$ and $\Sigma_2$ in ${\cal S}(\Gamma)$ are switching equivalent if and only if the symmetric difference of $E_{{\Sigma_1}^-}$ and $E_{{\Sigma_2}^-}$ is an edge cut of $\Gamma$.
\end{lem}

The \textit{generalized Petersen graph} ${\rm GP}(n, k)$, for $n\geqslant 3$ and $1\leqslant k\leqslant \lfloor (n-1)/2\rfloor$, is a cubic graph on the vertex set $V_{{\rm GP}(n, k)}=\mathbb{Z}_2\times\mathbb{Z}_n$,
for which the edge set is defined as follows: $$E_{{\rm GP}(n, k)}=\{(0, j)(0, j+1), \  (0, j)(1, j), \ (1, j)(1, j+k) \ | \  j=0, 1, \ldots, n-1\},$$ where all sums are taken modulo $n$.
These graphs were introduced by Coxeter \cite{Coxeter} and named by Watkins \cite{Watkins}.
For instance, the generalized Petersen graph ${\rm GP}(7,2)$ is depicted in Figure 1.
We usually call vertices $(0, 0), (0, 1), \ldots, (0, n-1)$ {\em outer vertices} and $(1, 0), (1, 1), \ldots, (1, n-1)$ {\em inner vertices}. Note that the outer and inner vertices are arranged on an outer circle and an inner circle, respectively.

\begin{center}
\tikzset{
    every node/.style={
        circle,
        draw,
        solid,
        fill=black!100,
        inner sep=2pt,
        minimum width=5pt
    }
}
\begin{tikzpicture}[thick,scale=0.7,shorten >=1pt, style/.style={draw,shape=circle}]
    \draw (0,3) node {} -- (0,1.5) node {};
    \draw (2.3,1.9) node {} -- (1.1,0.9)  node {};
    \draw (2.9,-0.6) node {} -- (1.4,-0.3)  node {};
    \draw (1.3,-2.7) node {} -- (0.6,-1.3)  node {};
    \draw (-2.3,1.9) node {} -- (-1.1,0.9) node {};
    \draw (-2.9,-0.6) node {} -- (-1.4,-0.3) node {};
    \draw (-1.3,-2.7) node {} -- (-0.6,-1.3)  node {};

   \draw (0,3) node {} -- (2.3,1.9)  node {};
   \draw (2.3,1.9) node {} -- (2.9,-0.6)  node {};
   \draw (2.9,-0.6) node {} -- (1.3,-2.7)  node {};
   \draw (1.3,-2.7) node {} -- (-1.3,-2.7)  node {};
   \draw (-2.3,1.9) node {} -- (-2.9,-0.6)  node {};
   \draw (-2.9,-0.6) node {} -- (-1.3,-2.7)  node {};
   \draw (-2.3,1.9) node {} -- (0,3) node {};

   \draw (0,1.5) node {} -- (-1.4,-0.3)  node {};
   \draw (-1.4,-0.3) node {} -- (0.6,-1.3)  node {};
   \draw (0.6,-1.3) node {} -- (1.1,0.9)  node {};
   \draw (1.1,0.9) node {} -- (-1.1,0.9)  node {};
   \draw (-1.1,0.9) node {} -- (-0.6,-1.3)  node {};
   \draw (-0.6,-1.3) node {} -- (1.4,-0.3)  node {};
   \draw (1.4,-0.3) node {} -- (0,1.5)  node {};
   \end{tikzpicture}
$$\textbf{Fig. 1. } \mbox{The generalized Petersen graph}  \ {\rm GP}(7,2)$$
\end{center}

We put $A(n, k)={\rm Aut}({\rm GP}(n,k))$. It is clear that $A(n, k)$ contains two automorphisms, $\alpha$ (the {\em rotation}) and $\beta$ (the {\em reflection}), defined by
$$\alpha: (i, j) \mapsto (i, j+1) \ \  \mbox{and} \  \beta: (i, j)\mapsto (i,-j), \ \ \mbox{ for} \ (i, j)\in \mathbb{Z}_2\times\mathbb{Z}_n.$$ Let $\gamma$ be the permutation of vertices defined by
$\gamma: (0, j)\mapsto (1, kj)$ and $\gamma: (1, j)\mapsto (0, kj)$.
Again, it is not difficult to see that $\gamma$ is an automorphism of ${\rm GP}(n, k)$ exactly when $k^2\equiv \pm1\pmod{n}$. The following theorem  (which is taken from \cite{FGW}), determines the  automorphism group of a generalized Petersen graph.
\begin{thm}\label{aut}  {\rm (see \cite{FGW})} Let $n$ and $k$ be positive integers and $(n, k)$ is not one of $(4, 1)$, $(5, 2)$, $(8, 3)$, $(10, 2)$, $(10, 3)$, $(12, 5)$ or $(24, 5)$. Then, the following statements hold:
 \begin{itemize}
 \item[\rm (a)] if $k^2\equiv 1\pmod{n}$, then we have  $$A(n,k) =\langle\alpha, \beta, \gamma \ | \ \alpha^n=\beta^2=\gamma^2=1,\alpha\beta=\beta\alpha^{-1},\alpha\gamma=\gamma\alpha^k,\beta\gamma=\gamma\beta\rangle.$$

 \item[\rm (b)] if $k^2\equiv -1\pmod{n}$, then we have $$A(n,k) =\langle\alpha,\beta,\gamma \ | \  \alpha^n=\beta^2=\gamma^2=1,\alpha\beta=\beta\alpha^{-1},\alpha\gamma=\gamma\alpha^k,\beta\gamma=\gamma\beta\rangle.$$

 \item[\rm (c)] if $k^2\not\equiv \pm1\pmod{n}$, then we have $A(n,k)=\langle\alpha,\beta \ | \ \alpha^n=\beta^2=1,\alpha\beta=\beta\alpha^{-1}\rangle$.
 \end{itemize}
\end{thm}

\subsection{The action}

 As is customary,  we denote by ${\rm Aut}(\Gamma)$  the group of all automorphisms of a graph $\Gamma$.
 An automorphism of a signed graph $\Sigma=(\Gamma, \sigma)$ is an automorphism of $\Gamma$  that
 preserves edge signs. The group of all automorphisms of a signed graph $\Sigma$ is denoted by ${\rm Aut}(\Sigma)$.

 We have the following easy observation (see Lemma 8.1 \cite{ZAS2}):
 $${\rm Aut}(\Sigma)={\rm Aut}(\Gamma)\cap {\rm Aut}(\Sigma^+)={\rm Aut}(\Gamma)\cap {\rm Aut}(\Sigma^-)={\rm Aut}(\Sigma^+) \cap {\rm Aut}(\Sigma^-).$$

 We now define an {\em action} of the automorphism group ${\rm Aut}(\Gamma)$ on the set $\Omega_s(\Gamma)$ by setting   $$[\Gamma, \sigma]^\varphi=[\Gamma, \sigma^\varphi],$$ where $$\sigma^\varphi(v_iv_j)=\sigma(\varphi^{-1}(v_i)\varphi^{-1}(v_j)).$$
We need to verify that this action is well defined. For that we prove the image of any two switching equivalent signed graphs, are also switching equivalent. If $(\Gamma, \sigma_1)$ and $(\Gamma, \sigma_2)$
are switching equivalent, then there exists a switching function $\theta$ such that $(\Gamma, \sigma_2)=(\Gamma, \sigma_1)^\theta=(\Gamma, \sigma_1^\theta)$, where as before $$\sigma_1^\theta(v_iv_j)=\theta(v_i) \sigma_1(v_iv_j) \theta(v_j),$$  for every $v_iv_j\in E_\Gamma$. We now consider the switching function $\theta \varphi^{-1}: V_\Gamma \rightarrow \{+1, -1\}$. We claim that
\begin{equation}\label{eee-3} (\Gamma, \sigma_1^\varphi)^{\theta\varphi^{-1}}=(\Gamma, \sigma_2^\varphi),\end{equation}
which means that the signed graphs $(\Gamma, \sigma_1^\varphi)$ and $(\Gamma, \sigma_2^\varphi)$ are switching equivalent and we are done. To prove (\ref{eee-3}), it suffices to show that
$$(\Gamma, (\sigma_1^\varphi)^{\theta\varphi^{-1}})=(\Gamma, \sigma_2^\varphi),$$
or equivalently,
$$ (\sigma_1^\varphi)^{\theta\varphi^{-1}}=\sigma_2^\varphi.$$
In fact,  for each $v_iv_j\in E_\Gamma$, we have
$$
\begin{array}{lll} (\sigma_1^\varphi)^{\theta\varphi^{-1}} (v_iv_j) & = & \theta\varphi^{-1} (v_i)  \sigma_1^\varphi (v_iv_j) \theta\varphi^{-1}(v_j)\\[0.1cm]
& = & \theta\varphi^{-1} (v_i)  \sigma_1 (\varphi^{-1}(v_i)\varphi^{-1}(v_j)) \theta\varphi^{-1}(v_j)\\[0.1cm]
& = & \sigma_1^\theta (\varphi^{-1}(v_i)\varphi^{-1}(v_j))\\[0.1cm]
& = & \sigma_2 (\varphi^{-1}(v_i)\varphi^{-1}(v_j))\\[0.1cm]
& = & \sigma_2^\varphi (v_iv_j),\\
\end{array}
$$ as required.

\begin{lem} \label{orbits} Let $\Gamma$ be a simple graph and let ${\rm Aut}(\Gamma)$ act on $\Omega_s(\Gamma)$ as above. Two signed graphs $(\Gamma, \sigma_1)$ and $(\Gamma, \sigma_2)$ are switching isomorphic if and only if $[\Gamma, \sigma_1]$ and $[\Gamma, \sigma_2]
$ belong to the same orbit. In particular, the number of switching non-isomorphic signed graphs is equal to the number of orbits of this group action.
\end{lem}
\begin{proof} Let $(\Gamma, \sigma_1)$ and $(\Gamma, \sigma_2)$ be two switching isomorphic graphs. By the definition, we may assume that $(\Gamma, \sigma_1)$ is isomorphic to $(\Gamma, \sigma_3)$ which is a switching equivalent to $(\Gamma, \sigma_2)$. Hence, there exists $\varphi\in {\rm Aut}(\Gamma)$, which preserves the sign of edges in $(\Gamma, \sigma_1)$, that is if $v_iv_j\in E_\Gamma$, then $$\sigma_3(v_iv_j)=\sigma_1(\varphi^{-1}(v_i)\varphi^{-1}(v_j)).$$
This shows that $\sigma_3=\sigma_1^\varphi$, and so
$$[\Gamma, \sigma_1]^\varphi=[\Gamma, \sigma_1^\varphi]=[\Gamma, \sigma_3]=[\Gamma, \sigma_2], $$
which means that  $[\Gamma, \sigma_1]$ and $[\Gamma, \sigma_2]$ are in the same orbit.

Conversely, if $[\Gamma, \sigma_1]$ and $[\Gamma, \sigma_2]$ both lie in the same orbit of  $\Omega_s(\Gamma)$ under ${\rm Aut}(\Gamma)$, then
there exists $\varphi\in {\rm Aut}(\Gamma)$ such that $[\Gamma, \sigma_1]=[\Gamma, \sigma_2]^\varphi$,
or equivalently, $[\Gamma, \sigma_1]=[\Gamma, \sigma_2^\varphi]$.
This means that $(\Gamma, \sigma_1)$ and $(\Gamma, \sigma_2^\varphi)$ are switching equivalent, and since $(\Gamma, \sigma_2^\varphi)\cong (\Gamma, \sigma_2)$ ($\varphi$ is a corresponding isomorphism), it follows that $(\Gamma, \sigma_1)$ and $(\Gamma, \sigma_2)$ are switching isomorphic. This completes the proof. \end{proof}

\begin{lem} \label{isonegative}  Let $\Sigma=(\Gamma, \sigma)$ be a signed graph and $\varphi\in {\rm Aut}(\Gamma)$.  Then ${\Sigma^\varphi}^-$ and $\Sigma^-$ are isomorphic.
\end{lem}
\begin{proof} We claim that $\varphi:V_\Gamma\rightarrow V_\Gamma$ is a graph isomorphism between ${\Sigma^\varphi}^-$ and $\Sigma^-$. Since $\varphi$ is a bijective mapping, it suffices to show that $v_iv_j\in E_{\Sigma^-}$ if and only if $\varphi(v_i)\varphi(v_j)\in E_{{\Sigma^\varphi}^-}$. Indeed, by definition of $\sigma^\varphi$, we have $\sigma^\varphi(\varphi(v_i)\varphi(v_j))=\sigma(v_iv_j)$, and hence $\sigma(v_iv_j)=-1$ if and only if
$\sigma^\varphi(\varphi(v_i)\varphi(v_j))=-1$,  in other words, $v_iv_j\in E_{\Sigma^-}$ if and only if $\varphi(v_i)\varphi(v_j)\in E_{{\Sigma^\varphi}^-}$, as required. \end{proof}


\section{The Main Results}
In the next result, we will determine a lower bound for the number of mutually switching non-isomorphic signed graphs on $n$ vertices with a complete underlying graph. We need to introduce some terminology. Let $\Sigma$ be a signed graph with vertex set $V_\Sigma=\{v_1, v_2, \ldots, v_n\}$. We denote by  $d^+_\Sigma(v_i)$ (resp. $d^-_\Sigma(v_i)$)  the number of positive (resp. negative) edges incident with $v_i$, and denote by $\psi (n, \Delta)$ the number of non-isomorphic graphs $\Gamma$ on $n$ vertices with $\Delta(\Gamma)\leqslant \Delta$.

\begin{thm}\label{lower} The number of mutually switching non-isomorphic signed graphs with a complete underlying graph on $n\geqslant 4$ vertices is at least $\psi(n, \lfloor \frac{n}{4}\rfloor-1)$.
\end{thm}
\begin{proof}  Let $n\geqslant 4$ be an integer and let ${\cal{A}}=\{\Gamma_1, \Gamma_2, \ldots, \Gamma_l\}$ be the set of all mutually non-isomorphic graphs on the vertex set $\{1, 2, \ldots, n\}$.
Clearly, every $\Gamma_i$ is a subgraph of $K_n$. It is obvious that two signed graphs $\Sigma_1$ and $\Sigma_2$ in ${\cal S}(K_n)$ are isomorphic if and only if the induced subgraphs $\Sigma_1^-$ and $\Sigma_2^-$ are isomorphic. For $i=1, 2, \ldots, l$, we define a function
 $\sigma_i: K_n\rightarrow \{-1, +1\}$ by
 $$\sigma_i(e)=\left\{\begin{array}{ll} -1 & \mbox{if} \ e\in E_{\Gamma_i},\\[0.1cm] +1 & \mbox{otherwise}.\end{array} \right. $$
 We claim the signed graphs $\{(K_n,\sigma_i) \ | \ i=1, 2, \ldots, l\}$ are mutually switching non-isomorphic.
 To prove the claim, we must show that for each $i\neq j$ of $\{1, 2, \ldots, l\}$, the signed graphs $(K_n, \sigma_i)$ and $(K_n, \sigma_j)$ cannot be switching isomorphic. First of all, note that the functions $\sigma_i$ have been chosen so that the signed graphs $(K_n, \sigma_i)$ and $(K_n, \sigma_j)$ cannot be isomorphic. We now prove that $(K_n, \sigma_j)$ cannot be isomorphic to any switching equivalent pair of $(K_n, \sigma_i)$.

 We make a few observations before going on to prove anything. Notice that any edge cut in $K_n$ induces a complete bipartite graph $K_{n_1, n_2}$ where $n=n_1+n_2$.   On the other hand, for each vertex $v$ in the graph $\Gamma_j$, $j=1, 2, \ldots, l$, we have
 \begin{equation}\label{e1} d^-_{(K_n, \sigma_j)}(v)=\deg_{\Gamma_j}(v)\leqslant \left \lfloor \frac{n}{4}\right\rfloor-1.\end{equation}
 Now we return to the proof of the claim. By way of contradiction we assume that the signed graph $(K_n, \sigma_j)$ is isomorphic to a signed graph $(K_n, \sigma)$ which is switching equivalent to $(K_n, \sigma_i)$.
Since $(K_n, \sigma)\sim (K_n, \sigma_i)$, by the definition there exists a switching function $\theta$ such that  $(K_n, \sigma)=(K_n, \sigma_i)^\theta$. Let $[S, \overline{S}]$ be a edge cut in $K_n$, which induces the complete bipartite graph $K_{n_1, n_2}$ where $n_1=|S|$ and $n_2=|\overline{S}|$. We may assume without loss of generality that $n_1\leqslant \lfloor\frac{n}{2}\rfloor\leqslant n_2$, hence, in the graph $K_{n_1, n_2}$, the degree of each vertex in $S$ is more than or equal to $\lfloor\frac{n}{2}\rfloor$. By Lemma \ref{HA} the symmetric difference of the edge sets $E_{(K_n,\sigma)^{-1}}, E_{(K_n,\sigma_i)^{-1}}$ induce an edge cut which is considered to be $K_{n_1,n_2}$. Let $\Sigma_i=(K_{n_1, n_2}, \hat{\sigma})$, where $\hat{\sigma}=\sigma_i|_{E(K_{n_1, n_2})}$. Now, it follows for every vertex $v$ in $S$ that
  $$d^+_{\Sigma_i}(v)=n_2-d^-_{\Sigma_i}(v)\geqslant n_2-d^-_{ (K_n, \sigma_i)  }(v)\geqslant n_2-\left \lfloor \frac{n}{4}\right\rfloor+1\geqslant \left\lfloor\frac{n}{2} \right\rfloor-\left \lfloor \frac{n}{4}\right\rfloor+1> \left \lfloor \frac{n}{4}\right\rfloor.$$
Since switching of a signed graph at a vertex set $S$ reverses the sign of each edge in $[S, \overline{S}]$,
we obtain
$$d^-_{\Sigma_i^\theta}(v)> \left \lfloor \frac{n}{4}\right\rfloor.$$
 On the other hand, singed graphs $(K_n, \sigma_j)$ and $(K_n, \sigma)$ are isomorphic, which implies that
  $$ d^-_{(K_n, \sigma_j)}(v)=d^-_{(K_n, \sigma)}(v)=d^-_{(K_n, \sigma_i)^\theta}(v)\geqslant d^-_{\Sigma_i^\theta}(v)> \left \lfloor \frac{n}{4}\right\rfloor,$$
which contradicts (\ref{e1}) and completes the proof.
\end{proof}

In what follows, we focus our attention on the generalized Petersen graphs $\Lambda={\rm GP}(p, k)$ where $p\geqslant 7$ is a prime and  $k^2\not\equiv \pm 1\pmod{p}$. It follows by Theorem \ref{aut}  that $${\rm Aut} (\Lambda)=A(p, k)=\langle\alpha,\beta \ | \ \alpha^p=\beta^2=1,\alpha\beta=\beta\alpha^{-1}\rangle,$$
which is a dihedral group of order $2p$. In group-theoretic terms, $A(p, k)$ is a Frobenius group with kernel ${\Bbb Z}_p$ and  complement ${\Bbb Z}_2$, hence $A(p, k)\cong {\Bbb Z}_p\rtimes {\Bbb Z}_2$.  Here we shall be interested only in deriving elementary properties
of automorphism groups of signed graphs on ${\rm GP}(p, k)$. To do this, we introduce a little more notation.

If $\Psi=(\Lambda, \sigma)$ is a signed graph on $\Lambda$, then for each $\emptyset \neq F\subseteq E_\Lambda$, define
$$\sigma(F):=\{\sigma(e) \ | \ e\in F\}\subseteq \{+1, -1\}.$$
Put $$\begin{array}{lll} E_1 &  \!\!\! :=  &  \!\!\!  \{(0, j)(0, j+1) \ | \ j=0, 1, \ldots, p-1\},\\[0.2cm]
E_2 &  \!\!\!  := &  \!\!\!  \{(1, j)(1, j+k) \ | \ j=0, 1, \ldots, p-1\},\\[0.2cm]
E_3 &  \!\!\!  := &  \!\!\!  \{(0, j)(1, j) \ | \ j=0, 1, \ldots, p-1\}.\\[0.2cm]
\end{array}$$

We make a preliminary observation:  $\langle \alpha\rangle$ acts transitively on $E_i$ for each $i$, $1\leqslant i\leqslant 3$.  For instance, suppose that $e_i=(0, i)(1, i+k)$ and $e_j=(0, j)(1, j+k)$ are two arbitrary elements of $E_3$ such that $j\geqslant i$. We now consider the rotation $\alpha^{j-i}$. It follows by the definition of $\alpha$ that
$$\alpha^{j-i}(0, i)=(0, j-i+i)=(0, j) \ \ \ \mbox{and} \ \ \ \alpha^{j-i}(1, i+k)=(1, j-i+i+k)=(1, j+k),$$
and so $\alpha^{j-i}(e_i)=e_j$. With this preliminary observation, we can now prove the following result.

\begin{thm}\label{auto-p}  With the above notation, if  $|\sigma(E_1)|= |\sigma(E_2)|= |\sigma(E_3)|=1$, then ${\rm Aut}(\Psi)=A(p, k)$, otherwise ${\rm Aut}(\Psi)$
contains at most two elements.
\end{thm}
\begin{proof} We already know that ${\rm Aut}(\Psi)=A(p, k)\cap {\rm Aut}(\Psi^-)$. Observe, first of all, that
$$A(p, k)=\{1, \alpha, \alpha^2, \ldots, \alpha^{p-1}, \beta, \alpha\beta, \alpha^2\beta, \ldots, \alpha^{p-1}\beta\},$$
all rotations $ \alpha, \alpha^2, \ldots, \alpha^{p-1}$ have order $p$ and all reflections $ \beta, \alpha\beta,  \ldots, \alpha^{p-1}\beta$ have order $2$.  Obviously, ${\rm Aut}(\Psi)$ as a subgroup of $A(p, k)$ has order $1$, $2$, $p$ or $2p$. Since $\langle \alpha\rangle$ acts transitively on $E_i$ for each $i$, none of these rotations is contained in $ {\rm Aut}(\Psi^-)$, except when $|\sigma(E_1)|= |\sigma(E_2)|= |\sigma(E_3)|=1$. As a matter of fact, in the case when $|\sigma(E_1)|= |\sigma(E_2)|= |\sigma(E_3)|=1$, since the sets $E_1$,  $E_2$ and $E_3$ are $\langle \alpha\rangle$-invariant,  $\langle\alpha \rangle\subseteq {\rm Aut}(\Psi^-)$. Similarly, each of these sets
is invariant under elements of the form $\alpha^i\beta$, $i=0, 1, \ldots, p-1$. Again it follows that if $|\sigma(E_1)|= |\sigma(E_2)|= |\sigma(E_3)|=1$, then ${\rm Aut}(\Psi^-)$ contains all elements of the form $\alpha^i\beta$, $i=0, 1, \ldots, p-1$. In this situation, we  therefore conclude that  ${\rm Aut}(\Psi^-)=A(p, k)={\rm Aut}(\Psi)$, as required.
\end{proof}

In what follows, we will focus our attention on signed graphs in which the underlying graph is a complete graph or a generalized Petersen graph. Let us first consider the signed graphs with the underlying graph $K_n$. In fact, the number of signed graphs with on $K_n$ is equal to $|{\cal S}(K_n)|=2^{n\choose 2}$, and
among all such signed graphs,  by
Lemma \ref{nds}, the number of mutually switching non-equivalent signed graphs with the underlying graph $K_n$ ie equal to
$$|\Omega_s(K_n)|=2^{{n\choose 2}-n+1}=2^{(n-2)(n-1)/2}.$$
The problem of finding the number of switching non-isomorphic graphs on $K_n$ is also useful as well as interesting. Recall that if we consider the action of ${\rm Aut}(K_n)$ on $\Omega_s(K_n)$ by $[K_n, \sigma]^\varphi=[K_n, \sigma^\varphi]$, where $\sigma^\varphi(v_iv_j)=\sigma(\varphi^{-1}(v_i)\varphi^{-1}(v_j))$, then by Lemma \ref{orbits},  the number of switching non-isomorphic graphs on $K_n$ is equal to the number of orbits of this action. It is well known that the automorphism group of the complete graph on $n$ vertices ${\rm Aut}(K_n)$ is isomorphic to $S_n$, which is $n$-fold transitive on the set $\{1, 2, \ldots, n\}$ for any $n$.
Clearly $n$-fold transitivity implies $k$-fold transitive for $1\leqslant k\leqslant n$, so any triply transitivity implies doubly transitive and transitive. Thus, when $n\geqslant 3$, ${\rm Aut}(K_n)\cong S_n$ is doubly transitive,
which means that ${\rm Aut}(K_n)$ is edge-transitive.

\begin{lem} \label{orbit}  Let ${\rm Aut}(K_n)$ act on $\Omega_s(K_n)$ as above, and $[K_n, \sigma]\in \Omega_s(K_n)$. Then, we have
$$[K_n, \sigma]^{{\rm Aut}(K_n)}=\left\{[K_n, \sigma'] \ | \ (K_n, \sigma')^-\cong (K_n, \sigma)^-\right\}.$$
\end{lem}
\begin{proof}  We put ${\cal O}:=[K_n, \sigma]^{{\rm Aut}(K_n)}$ and $\mathcal{A}:=\left\{[K_n, \sigma'] \ | \ (K_n, \sigma')^-\cong (K_n, \sigma)^-\right\}$. Assume first that $[K_n, \sigma'] \in \mathcal{A}$ and $\varphi: V_\Gamma \rightarrow V_\Gamma$ is a graph isomorphism between $(K_n, \sigma)^-$ and $(K_n, \sigma')^-$. Obviously $\varphi\in {\rm Aut}(K_n)$. We claim that
$[K_n, \sigma']=[K_n, \sigma]^\varphi\in \mathcal{O}$. It suffices to show that $(K_n, \sigma')=(K_n, \sigma)^\varphi$, or equivalently, $(K_n, \sigma')=(K_n, \sigma^\varphi)$. Note that the last equality is also equivalent to $\sigma'=\sigma^\varphi$. Since $\varphi: V_\Gamma \rightarrow V_\Gamma$ is a graph isomorphism between $(K_n, \sigma)^-$ and $(K_n, \sigma')^-$,  it follows that $v_iv_j\in E_{(K_n, \sigma)^-}$ if and only if $\varphi(v_i)\varphi(v_j)\in E_{(K_n, \sigma')^-}$, or equivalently,  $\sigma(v_iv_j)=-1$ if and only if $\sigma'(\varphi(v_i)\varphi(v_j))=-1$. Substituting $\varphi^{-1}(v_i)$ for $v_i$ and $\varphi^{-1}(v_j)$ for $v_j$  in the last expression, we obtain  $\sigma(\varphi^{-1}(v_i)\varphi^{-1}(v_j))=-1$ if and only if $\sigma'(v_iv_j)=-1$, or equivalently, $\sigma^\varphi(v_iv_j)=-1$ if and only if $\sigma'(v_iv_j)=-1$. This means that $\sigma^\varphi=\sigma'$, as required.

Assume next that $[K_n, \sigma]^\varphi\in \mathcal{O}$. To prove $[K_n, \sigma]^\varphi\in \mathcal{A}$, we must show that ${(K_n, \sigma)^\varphi}^-\cong (K_n, \sigma)^-$, or equivalently, $(K_n, \sigma^\varphi)^-\cong (K_n, \sigma)^-$. We now consider the automorphism $\varphi: V_\Gamma \rightarrow V_\Gamma$ and note that
if $v_iv_j\in E_{ (K_n, \sigma)^-}$, then $\sigma(v_iv_j)=-1$. But then it is immediate that
$$\sigma^\varphi(\varphi(v_i)\varphi(v_j))=\sigma\left(\varphi^{-1}(\varphi(v_i)) \varphi^{-1}(\varphi(v_j))\right)=\sigma(v_iv_j)=-1,$$
and so $\varphi(v_i)\varphi(v_j)\in E_{(K_n, \sigma^\varphi)^-}$. Obviously, the reverse direction also holds true. Therefore, we conclude that $\varphi$ is an isomorphism between $(K_n, \sigma^\varphi)^-$ and $(K_n, \sigma)^-$, as required. \end{proof}

\section{Some Applications}

In the following result, we focus our attention on a special case, that is $K_5$. We will use the following notation regarding the signed graphs:
$$\mu[\Gamma, \sigma]=\min_{\Sigma \in [\Gamma, \sigma]} |E_{\Sigma^-}|.$$
 In this section, the negative and positive edges are drawn in red and blue lines, respectively.

\begin{thm} \label{complete}  {\rm (see \cite{MS})} There are exactly seven signed graphs with the underlying graph $K_5$ up to switching isomorphism.
\end{thm}
\begin{proof}  First of all, we observe that by the preceding paragraph  $|{\cal S}(K_5)|=2^{10}$ and $|\Omega_s(K_5)|=2^{6}$.  Let ${\rm Aut}(K_5)$ act on $\Omega_s(K_5)$ by $[K_5, \sigma]^\varphi=[K_5, \sigma^\varphi]$, where $\sigma^\varphi(v_iv_j)=\sigma(\varphi^{-1}(v_i)\varphi^{-1}(v_j))$. We shall try to determine the number of orbits of this action instead of the number of switching non-isomorphic graphs on $K_5$.
More precisely, we shall  show that $2^6$ classes in $\Omega_s(K_5)$ are partitioned into the 7 orbits $\mathcal{O}_1, \mathcal{O}_2, \ldots, \mathcal{O}_7$, of size $1$, $10$, $15$, $12$, $15$, $10$ and $1$, respectively, corresponding to the signed graphs given in Figure 2.

\vspace{0.5cm}

\begin{tikzpicture}[mystyle/.style={draw,shape=circle}]
\def\ngon{5}
\node[regular polygon,regular polygon sides=\ngon,minimum size=2.5cm] (p) {};
\foreach\x in {1,...,\ngon}{\node[mystyle] (p\x) at (p.corner \x){};}
\foreach\x in {1,...,\numexpr\ngon-1\relax}{
\foreach\y in {\x,...,\ngon}{
\draw (p\x) -- (p\y)[blue];
\put(-18,-50){$\mathbf{(1).} \hspace{0.2cm} 1$}}}

\hspace{4cm}

\def\ngon{5}
\node[regular polygon,regular polygon sides=\ngon,minimum size=2.5cm] (p) {};
\foreach\x in {1,...,\ngon}{\node[mystyle] (p\x) at (p.corner \x){};}
\foreach\x in {1,...,\numexpr\ngon-1\relax}{
\foreach\y in {\x,...,\ngon}{
\draw (p\x) -- (p\y)[blue];
\draw (p1) -- (p2)[red];
\put(-18,-50){$\mathbf{(2).} \hspace{0.2cm} 10$}}}

\hspace{4cm}

\def\ngon{5}
\node[regular polygon,regular polygon sides=\ngon,minimum size=2.5cm] (p) {};
\foreach\x in {1,...,\ngon}{\node[mystyle] (p\x) at (p.corner \x){};}
\foreach\x in {1,...,\numexpr\ngon-1\relax}{
\foreach\y in {\x,...,\ngon}{
\draw (p\x) -- (p\y)[blue];
\draw (p1) -- (p2)[red];
\draw (p1) -- (p5)[red];
\put(-18,-50){$\mathbf{(3).} \hspace{0.2cm} 15$}}}

\hspace{4cm}

\def\ngon{5}
\node[regular polygon,regular polygon sides=\ngon,minimum size=2.5cm] (p) {};
\foreach\x in {1,...,\ngon}{\node[mystyle] (p\x) at (p.corner \x){};}
\foreach\x in {1,...,\numexpr\ngon-1\relax}{
\foreach\y in {\x,...,\ngon}{
\draw (p\x) -- (p\y)[blue];
\draw (p1) -- (p2)[red];
\draw (p1) -- (p5)[red];
\draw (p4) -- (p5)[red];
\put(-18,-50){$\mathbf{(4).} \hspace{0.2cm} 12$}}}

\end{tikzpicture}

\vspace{1cm}

\begin{tikzpicture}[mystyle/.style={draw,shape=circle}]
\hspace{2cm}

\def\ngon{5}
\node[regular polygon,regular polygon sides=\ngon,minimum size=2.5cm] (p) {};
\foreach\x in {1,...,\ngon}{\node[mystyle] (p\x) at (p.corner \x){};}
\foreach\x in {1,...,\numexpr\ngon-1\relax}{
\foreach\y in {\x,...,\ngon}{
\draw (p\x) -- (p\y)[blue];
\draw (p1) -- (p2)[red];
\draw (p4) -- (p5)[red];
\put(-18,-50){$\mathbf{(5).} \hspace{0.2cm} 15$}}}

\hspace{4cm}

\def\ngon{5}
\node[regular polygon,regular polygon sides=\ngon,minimum size=2.5cm] (p) {};
\foreach\x in {1,...,\ngon}{\node[mystyle] (p\x) at (p.corner \x){};}
\foreach\x in {1,...,\numexpr\ngon-1\relax}{
\foreach\y in {\x,...,\ngon}{
\draw (p\x) -- (p\y)[blue];
\draw (p1) -- (p2)[red];
\draw (p1) -- (p5)[red];
\draw (p2) -- (p5)[red];
\put(-17,-50){$\mathbf{(6).} \hspace{0.2cm} 10$}}}
\hspace{4cm}

\def\ngon{5}
\node[regular polygon,regular polygon sides=\ngon,minimum size=2.5cm] (p) {};
\foreach\x in {1,...,\ngon}{\node[mystyle] (p\x) at (p.corner \x){};}
\foreach\x in {1,...,\numexpr\ngon-1\relax}{
\foreach\y in {\x,...,\ngon}{
\draw (p\x) -- (p\y)[blue];
\draw (p1) -- (p2)[red];
\draw (p1) -- (p5)[red];
\draw (p2) -- (p5)[red];
\draw (p3) -- (p4)[red];
\put(-15,-50){$\mathbf{(7).} \hspace{0.2cm} 1$}}}
\end{tikzpicture}

\vspace{0.5cm}

\begin{center}
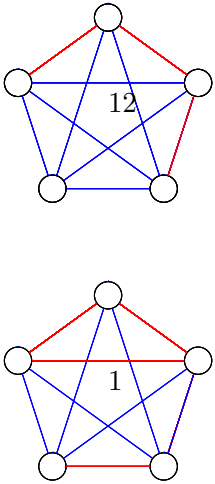
\textbf{Fig. 2.}    Seven switching non-isomorphic signed graphs on $K_5$.
\end{center}

Before we continue with the proof of this result we make some general remarks.
For signed graph $(K_5, \sigma_i)$, it follows that
\begin{equation}\label{e3} \begin{array}{lll}  \mathcal{O}_i= [K_5, \sigma_i]^{{\rm Aut}(K_5)} & = & \left\{[K_5, \sigma_i]^{\varphi} \ | \ \varphi\in {\rm Aut}(K_5)\right\}  \\[0.2cm]
 & = & \left\{[K_5, \sigma_i^{\varphi}] \ | \ \varphi\in {\rm Aut}(K_5)\right\} \ \ \ \ \ \ \mbox{(by the definition)}\\[0.2cm]
 & = &  \left\{[K_5, \sigma] \ | \ (K_5, \sigma)^-\cong (K_5, \sigma_i)^-\right\}. \ \ \ \mbox{(by Lemma \ref{orbit})}
\end{array}
\end{equation} Note that ${\rm Aut}(K_5)$ is isomorphic to $S_5$  which acts on the vertex-set $5$-transitive, and so
$K_5$ is edge-transitive, that is, if $v_iv_j, v_kv_l\in E_{K_5}$, then there exists an automorphism $\varphi$ in ${\rm Aut}(K_5)$ such that $\varphi(v_iv_j)=v_kv_l$.  We can now return and finish off our proof.  We treat separately the different cases:

\begin{itemize}

\item[{\rm (a)}]  $\mu[K_5, \sigma_1]=0$. The signed graph $(K_5, \sigma_1)$ is depicted in Fig. 2 (1). In this case,
$(K_5, \sigma_1)^-$ is the null graph, that is,  a graph with no edges. We see from (\ref{e3}) that  $\mathcal{O}_1$ contains
exactly one class $[K_5, \sigma_1]$.

\item[{\rm (b)}] $\mu[K_5, \sigma_2]=1$. The signed graph $(K_5, \sigma)$ is shown in Fig. 2 (2).  In this case,
$(K_5, \sigma_2)^-$ consists exactly of three isolated vertices and hence  $\mathcal{O}_2$ contains
exactly $10$ (one for each edge) class $[K_5, \sigma]$ for which $|\sigma^{-1}(-1)|=1$.

\item[{\rm (c)}] $\mu[K_5, \sigma_3]=2$. We distinguish between the cases: $(K_5, \sigma_3)^-$ contains two adjacent edges or two disjoint edges.
\begin{itemize}
\item[{\rm (c.1)}]  {\em $(K_5, \sigma_3)^-$ contains two adjacent edges.}  In this case, the signed graph appears as shown in Fig. 2 (3). Therefore, in veiw of (\ref{e3}), we need to count the number of signed graphs  $(K_5, \sigma)$ such that $(K_5, \sigma)^-\cong (K_5, \sigma_3)^-$. It is now easy to check that the number of such  signed graphs is $$\frac{1}{2}\left({5\choose 1}\times {4\choose 2}\right)=15.$$
Note that by switching the vertex of degree $2$ in  $(K_5, \sigma)^-$ we obtain a switching equivalent signed graph
of the same type.
\item[{\rm (c.2)}]   {\em $(K_5, \sigma_3)^-$ contains  two disjoint edges.}  The corresponding signed graph in this case  is as shown in Fig. 2 (5). As before, we conclude similarly that there are $\frac{1}{2}\left({5\choose 2}\times {3\choose 2}\right)=15$ signed graphs  $(K_5, \sigma)$ such that $(K_5, \sigma)^-\cong (K_5, \sigma_3)^-$.
\end{itemize}
\item[{\rm (d)}] $\mu[K_5, \sigma_4]=3$. Here, we distinguish among three cases: the number of vertices of degree $2$ is $3$, $2$ or $1$.

\begin{itemize}
\item[{\rm (d.1)}]  {\em $(K_5, \sigma_4)^-$ contains three vertices of degree $2$.}  The signed graph $(K_5, \sigma_4)$ in this case is shown in Fig. 2(6).  A routine argument shows that the number of signed graphs  $(K_5, \sigma)$ such that $(K_5, \sigma)^-\cong (K_5, \sigma_4)^-$  is ${5\choose 3}=10$.
\item[{\rm (d.2)}]   {\em $(K_5, \sigma_4)^-$ contains  two vertices of degree $2$.}  In this case, $(K_5, \sigma_4)^-$ is a path of length $3$ as shown in Fig. 2 (4). On the one hand, the number of such paths is $(5\times 4\times 3 \times 2)/2=60$. On the other hand,  if $v_1v_2v_3v_4$ is a path of length $3$, then by switching on $\{v_2\}$,  $\{v_3\}$, $\{v_1, v_3\}$ or $\{v_2, v_4\}$ we obtain a switching equivalent signed graph of the same type.
Hence, in this case, we have $60/5=12$ switching non-equivalent signed graphs $(K_5, \sigma)$ such that $(K_5, \sigma)^-\cong (K_5, \sigma_4)^-$.
\item[{\rm (d.3)}]  {\em $(K_5, \sigma_4)^-$ contains one vertices of degree $2$.}  In this case, by switching the vertex of degree $2$, we obtain a switching equivalent signed graph which has considered in (d.1).
\end{itemize}
\item[{\rm (e)}] $\mu[K_5, \sigma_5]=4$. The signed graph $(K_5, \sigma_5)$ is depicted in Fig. 2 (7).
If  $(K_5, \sigma)^-\cong (K_5, \sigma_5)^-$,  then it follows by switching at vertices of degree $1$ that both of
$(K_5, \sigma)$ and $(K_5, \sigma_5)$ are switching equivalent with $(K_5, \tau)$ for which $\tau^{-1}(-1)=E_{K_5}$.
Hence, we conclude that $[K_5, \sigma]=[K_5, \sigma_5]$,  which shows that $\mathcal{O}_7$ contains
exactly one class $[K_5, \sigma_5]$.
\end{itemize}

It follows from this discussion that the set $\Omega_s(K_5)$ can be partitioned into the 7 orbits $\mathcal{O}_1, \ldots, \mathcal{O}_7$, as required.
\end{proof}

Similarly, we have the following theorem.

\begin{thm}
There are exactly $36$ signed graphs with the underlying graph ${\rm GP}(7,2)$ up to switching isomorphism.
\end{thm}
\begin{proof} The proof is
quite similar to the proof of Theorem \ref{complete}, so we
avoid here full explanation of all details. Let $\Gamma={\rm GP}(7,2)$. As before, it follows that  $|{\cal S}(\Gamma)|=2^{21}$ and $|\Omega_s(\Gamma)|=2^{8}$.  Let ${\rm Aut}(\Gamma)$ act on $\Omega_s(\Gamma)$ by $[\Gamma, \sigma]^\varphi=[\Gamma, \sigma^\varphi]$. Again, we determine the number of orbits of this action instead of the number of switching non-isomorphic graphs on $\Gamma$.
As a matter of fact, the $2^8$ classes in $\Omega_s(\Gamma)$ can be partitioned into the $36$ orbits $\mathcal{O}_1, \mathcal{O}_2, \ldots, \mathcal{O}_{36}$, of size $1$, $7$, $7$, $7$, $7$, $7$, $7$, $7$, $7$, $14$, $7$, $7$,  $7$, $7$, $14$, $7$, $14$, $7$, $7$, $7$, $7$, $1$, $7$, $7$, $7$, $7$, $7$, $7$, $1$, $7$, $7$, $7$, $7$, $14$, $1$ and $7$, respectively, corresponding to the signed graphs given in Figure 3.\end{proof}

\begin{center}
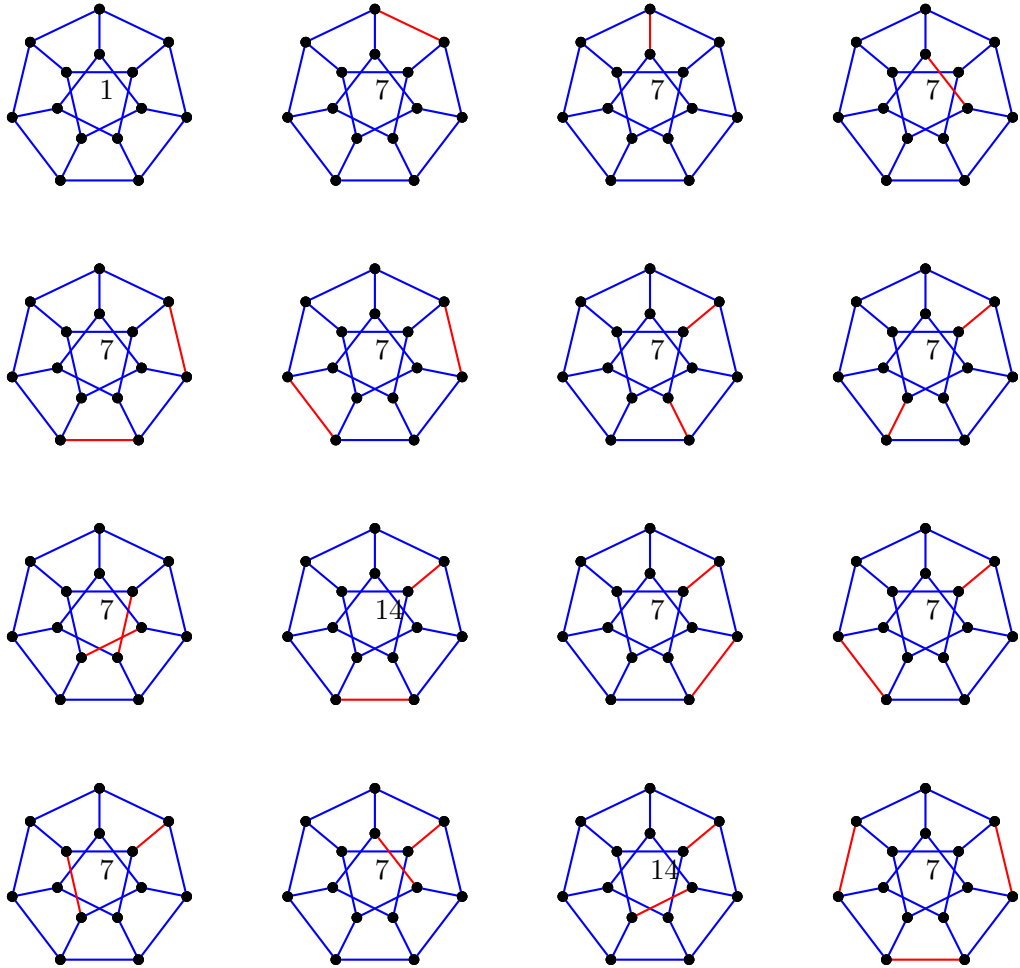
{\bf Fig. 3.}  36 switching non-isomorphic signed graphs on ${\rm GP}(7,2)$
\end{center}
\begin{center}


\tikzset{
    every node/.style={
        circle,
       fill=black!100,
        inner sep=1pt,
        minimum width=4pt
    }
}
\begin{scriptsize}

\begin{tikzpicture}[thick,scale=0.4,shorten >=2pt]
    \draw (0,3) node {} -- (0,1.5) [blue] node {};
    \draw (2.3,1.9) node {} -- (1.1,0.9) [blue] node {};
    \draw (2.9,-0.6) node {} -- (1.4,-0.3) [blue] node {};
    \draw (1.3,-2.7) node {} -- (0.6,-1.3) [blue] node {};
    \draw (-2.3,1.9) node {} -- (-1.1,0.9) [blue] node {};
    \draw (-2.9,-0.6) node {} -- (-1.4,-0.3) [blue] node {};
    \draw (-1.3,-2.7) node {} -- (-0.6,-1.3) [blue] node {};

   \draw (0,3) node {} -- (2.3,1.9) [blue] node {};
   \draw (2.3,1.9) node {} -- (2.9,-0.6) [blue] node {};
   \draw (2.9,-0.6) node {} -- (1.3,-2.7) [blue] node {};
   \draw (1.3,-2.7) node {} -- (-1.3,-2.7) [blue] node {};
   \draw (-2.3,1.9) node {} -- (-2.9,-0.6) [blue] node {};
   \draw (-2.9,-0.6) node {} -- (-1.3,-2.7) [blue] node {};
   \draw (-2.3,1.9) node {} -- (0,3) [blue] node {};

   \draw (0,1.5) node {} -- (-1.4,-0.3) [blue] node {};
   \draw (-1.4,-0.3) node {} -- (0.6,-1.3) [blue] node {};
   \draw (0.6,-1.3) node {} -- (1.1,0.9) [blue] node {};
   \draw (1.1,0.9) node {} -- (-1.1,0.9) [blue] node {};
   \draw (-1.1,0.9) node {} -- (-0.6,-1.3) [blue] node {};
   \draw (-0.6,-1.3) node {} -- (1.4,-0.3) [blue] node {};
   \draw (1.4,-0.3) node {} -- (0,1.5) [blue] node {};
   \put(-15,-45){$\mathbf{(1).} \hspace{0.2cm} 1$}
   \end{tikzpicture} \hspace{1cm}
\begin{tikzpicture}[thick,scale=0.4,shorten >=2pt]
    \draw (0,3) node {} -- (0,1.5) [blue] node {};
    \draw (2.3,1.9) node {} -- (1.1,0.9) [blue] node {};
    \draw (2.9,-0.6) node {} -- (1.4,-0.3) [blue] node {};
    \draw (1.3,-2.7) node {} -- (0.6,-1.3) [blue] node {};
    \draw (-2.3,1.9) node {} -- (-1.1,0.9) [blue] node {};
    \draw (-2.9,-0.6) node {} -- (-1.4,-0.3) [blue] node {};
    \draw (-1.3,-2.7) node {} -- (-0.6,-1.3) [blue] node {};

   \draw (0,3) node {} -- (2.3,1.9) [red] node {};%
   \draw (2.3,1.9) node {} -- (2.9,-0.6) [blue] node {};
   \draw (2.9,-0.6) node {} -- (1.3,-2.7) [blue] node {};
   \draw (1.3,-2.7) node {} -- (-1.3,-2.7) [blue] node {};
   \draw (-2.3,1.9) node {} -- (-2.9,-0.6) [blue] node {};
   \draw (-2.9,-0.6) node {} -- (-1.3,-2.7) [blue] node {};
   \draw (-2.3,1.9) node {} -- (0,3) [blue] node {};

   \draw (0,1.5) node {} -- (-1.4,-0.3) [blue] node {};
   \draw (-1.4,-0.3) node {} -- (0.6,-1.3) [blue] node {};
   \draw (0.6,-1.3) node {} -- (1.1,0.9) [blue] node {};
   \draw (1.1,0.9) node {} -- (-1.1,0.9) [blue] node {};
   \draw (-1.1,0.9) node {} -- (-0.6,-1.3) [blue] node {};
   \draw (-0.6,-1.3) node {} -- (1.4,-0.3) [blue] node {};
   \draw (1.4,-0.3) node {} -- (0,1.5) [blue] node {};
   \put(-15,-45){$\mathbf{(2).} \hspace{0.2cm} 7$}
   \end{tikzpicture} \hspace{1cm}
\begin{tikzpicture}[thick,scale=0.4,shorten >=2pt]
    \draw (0,3) node {} -- (0,1.5) [red] node {};%
    \draw (2.3,1.9) node {} -- (1.1,0.9) [blue] node {};
    \draw (2.9,-0.6) node {} -- (1.4,-0.3) [blue] node {};
    \draw (1.3,-2.7) node {} -- (0.6,-1.3) [blue] node {};
    \draw (-2.3,1.9) node {} -- (-1.1,0.9) [blue] node {};
    \draw (-2.9,-0.6) node {} -- (-1.4,-0.3) [blue] node {};
    \draw (-1.3,-2.7) node {} -- (-0.6,-1.3) [blue] node {};

   \draw (0,3) node {} -- (2.3,1.9) [blue] node {};
   \draw (2.3,1.9) node {} -- (2.9,-0.6) [blue] node {};
   \draw (2.9,-0.6) node {} -- (1.3,-2.7) [blue] node {};
   \draw (1.3,-2.7) node {} -- (-1.3,-2.7) [blue] node {};
   \draw (-2.3,1.9) node {} -- (-2.9,-0.6) [blue] node {};
   \draw (-2.9,-0.6) node {} -- (-1.3,-2.7) [blue] node {};
   \draw (-2.3,1.9) node {} -- (0,3) [blue] node {};

   \draw (0,1.5) node {} -- (-1.4,-0.3) [blue] node {};
   \draw (-1.4,-0.3) node {} -- (0.6,-1.3) [blue] node {};
   \draw (0.6,-1.3) node {} -- (1.1,0.9) [blue] node {};
   \draw (1.1,0.9) node {} -- (-1.1,0.9) [blue] node {};
   \draw (-1.1,0.9) node {} -- (-0.6,-1.3) [blue] node {};
   \draw (-0.6,-1.3) node {} -- (1.4,-0.3) [blue] node {};
   \draw (1.4,-0.3) node {} -- (0,1.5) [blue] node {};
    \put(-15,-45){$\mathbf{(3).} \hspace{0.2cm} 7$}
   \end{tikzpicture} \hspace{1cm}
\begin{tikzpicture}[thick,scale=0.4,shorten >=2pt]
    \draw (0,3) node {} -- (0,1.5) [blue] node {};
    \draw (2.3,1.9) node {} -- (1.1,0.9) [blue] node {};%
    \draw (2.9,-0.6) node {} -- (1.4,-0.3) [blue] node {};
    \draw (1.3,-2.7) node {} -- (0.6,-1.3) [blue] node {};
    \draw (-2.3,1.9) node {} -- (-1.1,0.9)  [blue] node {};
    \draw (-2.9,-0.6) node {} -- (-1.4,-0.3) [blue] node {};
    \draw (-1.3,-2.7) node {} -- (-0.6,-1.3) [blue] node {};

   \draw (0,3) node {} -- (2.3,1.9) [blue] node {};
   \draw (2.3,1.9) node {} -- (2.9,-0.6) [blue] node {};
   \draw (2.9,-0.6) node {} -- (1.3,-2.7) [blue] node {};
   \draw (1.3,-2.7) node {} -- (-1.3,-2.7) [blue] node {};
   \draw (-2.3,1.9) node {} -- (-2.9,-0.6) [blue] node {};
   \draw (-2.9,-0.6) node {} -- (-1.3,-2.7) [blue] node {};
   \draw (-2.3,1.9) node {} -- (0,3) [blue] node {};

   \draw (0,1.5) node {} -- (-1.4,-0.3) [blue] node {};
   \draw (-1.4,-0.3) node {} -- (0.6,-1.3) [blue] node {};
   \draw (0.6,-1.3) node {} -- (1.1,0.9) [blue] node {};
   \draw (1.1,0.9) node {} -- (-1.1,0.9) [blue] node {};
   \draw (-1.1,0.9) node {} -- (-0.6,-1.3) [blue] node {};
   \draw (-0.6,-1.3) node {} -- (1.4,-0.3) [blue] node {};
   \draw (1.4,-0.3) node {} -- (0,1.5) [red] node {};
    \put(-15,-45){$\mathbf{(4).} \hspace{0.2cm} 7$}
   \end{tikzpicture} \\ \vspace{1cm}

\begin{tikzpicture}[thick,scale=0.4,shorten >=2pt]
    \draw (0,3) node {} -- (0,1.5) [blue] node {};
    \draw (2.3,1.9) node {} -- (1.1,0.9) [blue] node {};
    \draw (2.9,-0.6) node {} -- (1.4,-0.3) [blue] node {};
    \draw (1.3,-2.7) node {} -- (0.6,-1.3) [blue] node {};
    \draw (-2.3,1.9) node {} -- (-1.1,0.9) [blue] node {};
    \draw (-2.9,-0.6) node {} -- (-1.4,-0.3) [blue] node {};
    \draw (-1.3,-2.7) node {} -- (-0.6,-1.3) [blue] node {};

   \draw (0,3) node {} -- (2.3,1.9) [blue] node {};
   \draw (2.3,1.9) node {} -- (2.9,-0.6) [red] node {};%
   \draw (2.9,-0.6) node {} -- (1.3,-2.7) [blue] node {};
   \draw (1.3,-2.7) node {} -- (-1.3,-2.7) [red] node {};%
   \draw (-2.3,1.9) node {} -- (-2.9,-0.6) [blue] node {};
   \draw (-2.9,-0.6) node {} -- (-1.3,-2.7) [blue] node {};
   \draw (-2.3,1.9) node {} -- (0,3) [blue] node {};

   \draw (0,1.5) node {} -- (-1.4,-0.3) [blue] node {};
   \draw (-1.4,-0.3) node {} -- (0.6,-1.3) [blue] node {};
   \draw (0.6,-1.3) node {} -- (1.1,0.9) [blue] node {};
   \draw (1.1,0.9) node {} -- (-1.1,0.9) [blue] node {};
   \draw (-1.1,0.9) node {} -- (-0.6,-1.3) [blue] node {};
   \draw (-0.6,-1.3) node {} -- (1.4,-0.3) [blue] node {};
   \draw (1.4,-0.3) node {} -- (0,1.5) [blue] node {};
   \put(-15,-45){$\mathbf{(5).} \hspace{0.2cm} 7$}
   \end{tikzpicture} \hspace{1cm}
\begin{tikzpicture}[thick,scale=0.4,shorten >=2pt]
    \draw (0,3) node {} -- (0,1.5) [blue] node {};
    \draw (2.3,1.9) node {} -- (1.1,0.9) [blue] node {};
    \draw (2.9,-0.6) node {} -- (1.4,-0.3) [blue] node {};
    \draw (1.3,-2.7) node {} -- (0.6,-1.3) [blue] node {};
    \draw (-2.3,1.9) node {} -- (-1.1,0.9) [blue] node {};
    \draw (-2.9,-0.6) node {} -- (-1.4,-0.3) [blue] node {};
    \draw (-1.3,-2.7) node {} -- (-0.6,-1.3) [blue] node {};

   \draw (0,3) node {} -- (2.3,1.9) [blue] node {};
   \draw (2.3,1.9) node {} -- (2.9,-0.6) [red] node {};%
   \draw (2.9,-0.6) node {} -- (1.3,-2.7) [blue] node {};
   \draw (1.3,-2.7) node {} -- (-1.3,-2.7) [blue] node {};
   \draw (-2.3,1.9) node {} -- (-2.9,-0.6) [blue] node {};
   \draw (-2.9,-0.6) node {} -- (-1.3,-2.7) [red] node {};%
   \draw (-2.3,1.9) node {} -- (0,3) [blue] node {};

   \draw (0,1.5) node {} -- (-1.4,-0.3) [blue] node {};
   \draw (-1.4,-0.3) node {} -- (0.6,-1.3) [blue] node {};
   \draw (0.6,-1.3) node {} -- (1.1,0.9) [blue] node {};
   \draw (1.1,0.9) node {} -- (-1.1,0.9) [blue] node {};
   \draw (-1.1,0.9) node {} -- (-0.6,-1.3) [blue] node {};
   \draw (-0.6,-1.3) node {} -- (1.4,-0.3) [blue] node {};
   \draw (1.4,-0.3) node {} -- (0,1.5) [blue] node {};
   \put(-15,-45){$\mathbf{(6).} \hspace{0.2cm} 7$}
   \end{tikzpicture} \hspace{1cm}
\begin{tikzpicture}[thick,scale=0.4,shorten >=2pt]
    \draw (0,3) node {} -- (0,1.5) [blue] node {};
    \draw (2.3,1.9) node {} -- (1.1,0.9) [red] node {};%
    \draw (2.9,-0.6) node {} -- (1.4,-0.3) [blue] node {};
    \draw (1.3,-2.7) node {} -- (0.6,-1.3) [red] node {};  %
    \draw (-2.3,1.9) node {} -- (-1.1,0.9) [blue] node {};
    \draw (-2.9,-0.6) node {} -- (-1.4,-0.3) [blue] node {};
    \draw (-1.3,-2.7) node {} -- (-0.6,-1.3) [blue] node {};

   \draw (0,3) node {} -- (2.3,1.9) [blue] node {};
   \draw (2.3,1.9) node {} -- (2.9,-0.6) [blue] node {};
   \draw (2.9,-0.6) node {} -- (1.3,-2.7) [blue] node {};
   \draw (1.3,-2.7) node {} -- (-1.3,-2.7) [blue] node {};
   \draw (-2.3,1.9) node {} -- (-2.9,-0.6) [blue] node {};
   \draw (-2.9,-0.6) node {} -- (-1.3,-2.7) [blue] node {};
   \draw (-2.3,1.9) node {} -- (0,3) [blue] node {};

   \draw (0,1.5) node {} -- (-1.4,-0.3) [blue] node {};
   \draw (-1.4,-0.3) node {} -- (0.6,-1.3) [blue] node {};
   \draw (0.6,-1.3) node {} -- (1.1,0.9) [blue] node {};
   \draw (1.1,0.9) node {} -- (-1.1,0.9) [blue] node {};
   \draw (-1.1,0.9) node {} -- (-0.6,-1.3) [blue] node {};
   \draw (-0.6,-1.3) node {} -- (1.4,-0.3) [blue] node {};
   \draw (1.4,-0.3) node {} -- (0,1.5) [blue] node {};
   \put(-15,-45){$\mathbf{(7).} \hspace{0.2cm} 7$}
   \end{tikzpicture} \hspace{1cm}
\begin{tikzpicture}[thick,scale=0.4,shorten >=2pt]
    \draw (0,3) node {} -- (0,1.5) [blue] node {};
    \draw (2.3,1.9) node {} -- (1.1,0.9) [red] node {};%
    \draw (2.9,-0.6) node {} -- (1.4,-0.3) [blue] node {};
    \draw (1.3,-2.7) node {} -- (0.6,-1.3) [blue] node {};
    \draw (-2.3,1.9) node {} -- (-1.1,0.9)  [blue] node {};
    \draw (-2.9,-0.6) node {} -- (-1.4,-0.3) [blue] node {};
    \draw (-1.3,-2.7) node {} -- (-0.6,-1.3) [red] node {};%

   \draw (0,3) node {} -- (2.3,1.9) [blue] node {};
   \draw (2.3,1.9) node {} -- (2.9,-0.6) [blue] node {};
   \draw (2.9,-0.6) node {} -- (1.3,-2.7) [blue] node {};
   \draw (1.3,-2.7) node {} -- (-1.3,-2.7) [blue] node {};
   \draw (-2.3,1.9) node {} -- (-2.9,-0.6) [blue] node {};
   \draw (-2.9,-0.6) node {} -- (-1.3,-2.7) [blue] node {};
   \draw (-2.3,1.9) node {} -- (0,3) [blue] node {};

   \draw (0,1.5) node {} -- (-1.4,-0.3) [blue] node {};
   \draw (-1.4,-0.3) node {} -- (0.6,-1.3) [blue] node {};
   \draw (0.6,-1.3) node {} -- (1.1,0.9) [blue] node {};
   \draw (1.1,0.9) node {} -- (-1.1,0.9) [blue] node {};
   \draw (-1.1,0.9) node {} -- (-0.6,-1.3) [blue] node {};
   \draw (-0.6,-1.3) node {} -- (1.4,-0.3) [blue] node {};
   \draw (1.4,-0.3) node {} -- (0,1.5) [blue] node {};
   \put(-15,-45){$\mathbf{(8).} \hspace{0.2cm} 7$}
   \end{tikzpicture} \\ \vspace{1cm}
\begin{tikzpicture}[thick,scale=0.4,shorten >=2pt]
    \draw (0,3) node {} -- (0,1.5) [blue] node {};
    \draw (2.3,1.9) node {} -- (1.1,0.9) [blue] node {};
    \draw (2.9,-0.6) node {} -- (1.4,-0.3) [blue] node {};
    \draw (1.3,-2.7) node {} -- (0.6,-1.3) [blue] node {};
    \draw (-2.3,1.9) node {} -- (-1.1,0.9) [blue] node {};
    \draw (-2.9,-0.6) node {} -- (-1.4,-0.3) [blue] node {};
    \draw (-1.3,-2.7) node {} -- (-0.6,-1.3) [blue] node {};

   \draw (0,3) node {} -- (2.3,1.9) [blue] node {};
   \draw (2.3,1.9) node {} -- (2.9,-0.6) [blue] node {};
   \draw (2.9,-0.6) node {} -- (1.3,-2.7) [blue] node {};
   \draw (1.3,-2.7) node {} -- (-1.3,-2.7) [blue] node {};
   \draw (-2.3,1.9) node {} -- (-2.9,-0.6) [blue] node {};
   \draw (-2.9,-0.6) node {} -- (-1.3,-2.7) [blue] node {};
   \draw (-2.3,1.9) node {} -- (0,3) [blue] node {};

   \draw (0,1.5) node {} -- (-1.4,-0.3) [blue] node {};
   \draw (-1.4,-0.3) node {} -- (0.6,-1.3)  [blue] node {};
   \draw (0.6,-1.3) node {} -- (1.1,0.9) [red] node {};%
   \draw (1.1,0.9) node {} -- (-1.1,0.9) [blue] node {};
   \draw (-1.1,0.9) node {} -- (-0.6,-1.3) [blue] node {};
   \draw (-0.6,-1.3) node {} -- (1.4,-0.3) [red] node {};%
   \draw (1.4,-0.3) node {} -- (0,1.5) [blue] node {};
   \put(-15,-45){$\mathbf{(9).} \hspace{0.2cm} 7$}
   \end{tikzpicture} \hspace{1cm}
\begin{tikzpicture}[thick,scale=0.4,shorten >=2pt]
    \draw (0,3) node {} -- (0,1.5) [blue] node {};
    \draw (2.3,1.9) node {} -- (1.1,0.9)[red] node {};%
    \draw (2.9,-0.6) node {} -- (1.4,-0.3) [blue] node {};
    \draw (1.3,-2.7) node {} -- (0.6,-1.3) [blue] node {};
    \draw (-2.3,1.9) node {} -- (-1.1,0.9) [blue] node {};
    \draw (-2.9,-0.6) node {} -- (-1.4,-0.3) [blue] node {};
    \draw (-1.3,-2.7) node {} -- (-0.6,-1.3) [blue] node {};

   \draw (0,3) node {} -- (2.3,1.9) [blue] node {};
   \draw (2.3,1.9) node {} -- (2.9,-0.6) [blue] node {};
   \draw (2.9,-0.6) node {} -- (1.3,-2.7) [blue] node {};
   \draw (1.3,-2.7) node {} -- (-1.3,-2.7) [red] node {};%
   \draw (-2.3,1.9) node {} -- (-2.9,-0.6) [blue] node {};
   \draw (-2.9,-0.6) node {} -- (-1.3,-2.7) [blue] node {};
   \draw (-2.3,1.9) node {} -- (0,3) [blue] node {};

   \draw (0,1.5) node {} -- (-1.4,-0.3) [blue] node {};
   \draw (-1.4,-0.3) node {} -- (0.6,-1.3) [blue] node {};
   \draw (0.6,-1.3) node {} -- (1.1,0.9) [blue] node {};
   \draw (1.1,0.9) node {} -- (-1.1,0.9) [blue] node {};
   \draw (-1.1,0.9) node {} -- (-0.6,-1.3) [blue] node {};
   \draw (-0.6,-1.3) node {} -- (1.4,-0.3) [blue] node {};
   \draw (1.4,-0.3) node {} -- (0,1.5) [blue] node {};
   \put(-15,-45){$\mathbf{(10).} \hspace{0.2cm} 14$}
   \end{tikzpicture} \hspace{1cm}
\begin{tikzpicture}[thick,scale=0.4,shorten >=2pt]
    \draw (0,3) node {} -- (0,1.5) [blue] node {};
    \draw (2.3,1.9) node {} -- (1.1,0.9) [red] node {};
    \draw (2.9,-0.6) node {} -- (1.4,-0.3) [blue] node {};
    \draw (1.3,-2.7) node {} -- (0.6,-1.3) [blue] node {};
    \draw (-2.3,1.9) node {} -- (-1.1,0.9) [blue] node {};
    \draw (-2.9,-0.6) node {} -- (-1.4,-0.3) [blue] node {};
    \draw (-1.3,-2.7) node {} -- (-0.6,-1.3) [blue] node {};

   \draw (0,3) node {} -- (2.3,1.9) [blue] node {};
   \draw (2.3,1.9) node {} -- (2.9,-0.6) [blue] node {};
   \draw (2.9,-0.6) node {} -- (1.3,-2.7) [red] node {};
   \draw (1.3,-2.7) node {} -- (-1.3,-2.7) [blue] node {};
   \draw (-2.3,1.9) node {} -- (-2.9,-0.6) [blue] node {};
   \draw (-2.9,-0.6) node {} -- (-1.3,-2.7) [blue] node {};
   \draw (-2.3,1.9) node {} -- (0,3) [blue] node {};

   \draw (0,1.5) node {} -- (-1.4,-0.3) [blue] node {};
   \draw (-1.4,-0.3) node {} -- (0.6,-1.3) [blue] node {};
   \draw (0.6,-1.3) node {} -- (1.1,0.9) [blue] node {};
   \draw (1.1,0.9) node {} -- (-1.1,0.9) [blue] node {};
   \draw (-1.1,0.9) node {} -- (-0.6,-1.3) [blue] node {};
   \draw (-0.6,-1.3) node {} -- (1.4,-0.3) [blue] node {};
   \draw (1.4,-0.3) node {} -- (0,1.5) [blue] node {};
    \put(-15,-45){$\mathbf{(11).} \hspace{0.2cm} 7$}
   \end{tikzpicture} \hspace{1cm}
\begin{tikzpicture}[thick,scale=0.4,shorten >=2pt]
    \draw (0,3) node {} -- (0,1.5) [blue] node {};
    \draw (2.3,1.9) node {} -- (1.1,0.9) [red] node {};%
    \draw (2.9,-0.6) node {} -- (1.4,-0.3) [blue] node {};
    \draw (1.3,-2.7) node {} -- (0.6,-1.3) [blue] node {};
    \draw (-2.3,1.9) node {} -- (-1.1,0.9) [blue] node {};
    \draw (-2.9,-0.6) node {} -- (-1.4,-0.3) [blue] node {};
    \draw (-1.3,-2.7) node {} -- (-0.6,-1.3) [blue] node {};

   \draw (0,3) node {} -- (2.3,1.9) [blue] node {};
   \draw (2.3,1.9) node {} -- (2.9,-0.6) [blue] node {};
   \draw (2.9,-0.6) node {} -- (1.3,-2.7) [blue] node {};
   \draw (1.3,-2.7) node {} -- (-1.3,-2.7) [blue] node {};
   \draw (-2.3,1.9) node {} -- (-2.9,-0.6)  [blue] node {};
   \draw (-2.9,-0.6) node {} -- (-1.3,-2.7) [red] node {};%
   \draw (-2.3,1.9) node {} -- (0,3) [blue] node {};

   \draw (0,1.5) node {} -- (-1.4,-0.3) [blue] node {};
   \draw (-1.4,-0.3) node {} -- (0.6,-1.3) [blue] node {};
   \draw (0.6,-1.3) node {} -- (1.1,0.9) [blue] node {};
   \draw (1.1,0.9) node {} -- (-1.1,0.9) [blue] node {};
   \draw (-1.1,0.9) node {} -- (-0.6,-1.3) [blue] node {};
   \draw (-0.6,-1.3) node {} -- (1.4,-0.3) [blue] node {};
   \draw (1.4,-0.3) node {} -- (0,1.5) [blue] node {};
    \put(-15,-45){$\mathbf{(12).} \hspace{0.2cm} 7$}
   \end{tikzpicture} \\ \vspace{1cm}
\begin{tikzpicture}[thick,scale=0.4,shorten >=2pt]
    \draw (0,3) node {} -- (0,1.5) [blue] node {};
    \draw (2.3,1.9) node {} -- (1.1,0.9) [red] node {};%
    \draw (2.9,-0.6) node {} -- (1.4,-0.3) [blue] node {};
    \draw (1.3,-2.7) node {} -- (0.6,-1.3) [blue] node {};
    \draw (-2.3,1.9) node {} -- (-1.1,0.9) [blue] node {};
    \draw (-2.9,-0.6) node {} -- (-1.4,-0.3) [blue] node {};
    \draw (-1.3,-2.7) node {} -- (-0.6,-1.3) [blue] node {};

   \draw (0,3) node {} -- (2.3,1.9) [blue] node {};
   \draw (2.3,1.9) node {} -- (2.9,-0.6) [blue] node {};
   \draw (2.9,-0.6) node {} -- (1.3,-2.7) [blue] node {};
   \draw (1.3,-2.7) node {} -- (-1.3,-2.7) [blue] node {};
   \draw (-2.3,1.9) node {} -- (-2.9,-0.6) [blue] node {};
   \draw (-2.9,-0.6) node {} -- (-1.3,-2.7) [blue] node {};
   \draw (-2.3,1.9) node {} -- (0,3) [blue] node {};

   \draw (0,1.5) node {} -- (-1.4,-0.3) [blue] node {};
   \draw (-1.4,-0.3) node {} -- (0.6,-1.3) [blue] node {};
   \draw (0.6,-1.3) node {} -- (1.1,0.9) [blue] node {};
   \draw (1.1,0.9) node {} -- (-1.1,0.9) [blue] node {};
   \draw (-1.1,0.9) node {} -- (-0.6,-1.3) [red] node {};%
   \draw (-0.6,-1.3) node {} -- (1.4,-0.3) [blue] node {};
   \draw (1.4,-0.3) node {} -- (0,1.5) [blue] node {};
    \put(-15,-45){$\mathbf{(13).} \hspace{0.2cm} 7$}
   \end{tikzpicture} \hspace{1cm}
\begin{tikzpicture}[thick,scale=0.4,shorten >=2pt]
    \draw (0,3) node {} -- (0,1.5) [blue] node {};
    \draw (2.3,1.9) node {} -- (1.1,0.9) [red] node {};%
    \draw (2.9,-0.6) node {} -- (1.4,-0.3) [blue] node {};
    \draw (1.3,-2.7) node {} -- (0.6,-1.3) [blue] node {};
    \draw (-2.3,1.9) node {} -- (-1.1,0.9) [blue] node {};
    \draw (-2.9,-0.6) node {} -- (-1.4,-0.3) [blue] node {};
    \draw (-1.3,-2.7) node {} -- (-0.6,-1.3) [blue] node {};

   \draw (0,3) node {} -- (2.3,1.9) [blue] node {};
   \draw (2.3,1.9) node {} -- (2.9,-0.6) [blue] node {};
   \draw (2.9,-0.6) node {} -- (1.3,-2.7) [blue] node {};
   \draw (1.3,-2.7) node {} -- (-1.3,-2.7) [blue] node {};
   \draw (-2.3,1.9) node {} -- (-2.9,-0.6) [blue] node {};
   \draw (-2.9,-0.6) node {} -- (-1.3,-2.7) [blue] node {};
   \draw (-2.3,1.9) node {} -- (0,3) [blue] node {};

   \draw (0,1.5) node {} -- (-1.4,-0.3) [blue] node {};
   \draw (-1.4,-0.3) node {} -- (0.6,-1.3) [blue] node {};
   \draw (0.6,-1.3) node {} -- (1.1,0.9) [blue] node {};
   \draw (1.1,0.9) node {} -- (-1.1,0.9) [blue] node {};
   \draw (-1.1,0.9) node {} -- (-0.6,-1.3) [blue] node {};
   \draw (-0.6,-1.3) node {} -- (1.4,-0.3) [blue] node {};
   \draw (1.4,-0.3) node {} -- (0,1.5) [red] node {};%
    \put(-15,-45){$\mathbf{(14).} \hspace{0.2cm} 7$}
   \end{tikzpicture} \hspace{1cm}
\begin{tikzpicture}[thick,scale=0.4,shorten >=2pt]
    \draw (0,3) node {} -- (0,1.5) [blue] node {};
    \draw (2.3,1.9) node {} -- (1.1,0.9) [red] node {};%
    \draw (2.9,-0.6) node {} -- (1.4,-0.3) [blue] node {};
    \draw (1.3,-2.7) node {} -- (0.6,-1.3) [blue] node {};
    \draw (-2.3,1.9) node {} -- (-1.1,0.9) [blue] node {};
    \draw (-2.9,-0.6) node {} -- (-1.4,-0.3) [blue] node {};
    \draw (-1.3,-2.7) node {} -- (-0.6,-1.3) [blue] node {};

   \draw (0,3) node {} -- (2.3,1.9) [blue] node {};
   \draw (2.3,1.9) node {} -- (2.9,-0.6) [blue] node {};
   \draw (2.9,-0.6) node {} -- (1.3,-2.7) [blue] node {};
   \draw (1.3,-2.7) node {} -- (-1.3,-2.7) [blue] node {};
   \draw (-2.3,1.9) node {} -- (-2.9,-0.6) [blue] node {};
   \draw (-2.9,-0.6) node {} -- (-1.3,-2.7) [blue] node {};
   \draw (-2.3,1.9) node {} -- (0,3) [blue] node {};

   \draw (0,1.5) node {} -- (-1.4,-0.3) [blue] node {};
   \draw (-1.4,-0.3) node {} -- (0.6,-1.3) [blue] node {};
   \draw (0.6,-1.3) node {} -- (1.1,0.9) [blue] node {};
   \draw (1.1,0.9) node {} -- (-1.1,0.9) [blue] node {};
   \draw (-1.1,0.9) node {} -- (-0.6,-1.3) [blue] node {};
   \draw (-0.6,-1.3) node {} -- (1.4,-0.3) [red] node {};%
   \draw (1.4,-0.3) node {} -- (0,1.5) [blue] node {};
    \put(-15,-45){$\mathbf{(15).} \hspace{0.2cm} 14$}
   \end{tikzpicture} \hspace{1cm}
\begin{tikzpicture}[thick,scale=0.4,shorten >=2pt]
    \draw (0,3) node {} -- (0,1.5) [blue] node {};
    \draw (2.3,1.9) node {} -- (1.1,0.9) [blue] node {};
    \draw (2.9,-0.6) node {} -- (1.4,-0.3) [blue] node {};
    \draw (1.3,-2.7) node {} -- (0.6,-1.3) [blue] node {};
    \draw (-2.3,1.9) node {} -- (-1.1,0.9) [blue] node {};
    \draw (-2.9,-0.6) node {} -- (-1.4,-0.3) [blue] node {};
    \draw (-1.3,-2.7) node {} -- (-0.6,-1.3) [blue] node {};

   \draw (0,3) node {} -- (2.3,1.9) [blue] node {};
   \draw (2.3,1.9) node {} -- (2.9,-0.6) [red] node {};%
   \draw (2.9,-0.6) node {} -- (1.3,-2.7) [blue] node {};
   \draw (1.3,-2.7) node {} -- (-1.3,-2.7) [red] node {};%
   \draw (-2.3,1.9) node {} -- (-2.9,-0.6) [red] node {};%
   \draw (-2.9,-0.6) node {} -- (-1.3,-2.7) [blue] node {};
   \draw (-2.3,1.9) node {} -- (0,3) [blue] node {};

   \draw (0,1.5) node {} -- (-1.4,-0.3) [blue] node {};
   \draw (-1.4,-0.3) node {} -- (0.6,-1.3) [blue] node {};
   \draw (0.6,-1.3) node {} -- (1.1,0.9) [blue] node {};
   \draw (1.1,0.9) node {} -- (-1.1,0.9) [blue] node {};
   \draw (-1.1,0.9) node {} -- (-0.6,-1.3) [blue] node {};
   \draw (-0.6,-1.3) node {} -- (1.4,-0.3) [blue] node {};
   \draw (1.4,-0.3) node {} -- (0,1.5) [blue] node {};
    \put(-15,-45){$\mathbf{(16).} \hspace{0.2cm} 7$}
   \end{tikzpicture} \\ \vspace{1cm}
\begin{tikzpicture}[thick,scale=0.4,shorten >=2pt]
    \draw (0,3) node {} -- (0,1.5) [red] node {};%
    \draw (2.3,1.9) node {} -- (1.1,0.9) [blue] node {};
    \draw (2.9,-0.6) node {} -- (1.4,-0.3) [blue] node {};
    \draw (1.3,-2.7) node {} -- (0.6,-1.3) [blue] node {};
    \draw (-2.3,1.9) node {} -- (-1.1,0.9) [blue] node {};
    \draw (-2.9,-0.6) node {} -- (-1.4,-0.3) [blue] node {};
    \draw (-1.3,-2.7) node {} -- (-0.6,-1.3) [blue] node {};

   \draw (0,3) node {} -- (2.3,1.9) [blue] node {};
   \draw (2.3,1.9) node {} -- (2.9,-0.6) [red] node {};%
   \draw (2.9,-0.6) node {} -- (1.3,-2.7) [blue] node {};
   \draw (1.3,-2.7) node {} -- (-1.3,-2.7) [red] node {};%
   \draw (-2.3,1.9) node {} -- (-2.9,-0.6) [blue] node {};
   \draw (-2.9,-0.6) node {} -- (-1.3,-2.7) [blue] node {};
   \draw (-2.3,1.9) node {} -- (0,3) [blue] node {};

   \draw (0,1.5) node {} -- (-1.4,-0.3) [blue] node {};
   \draw (-1.4,-0.3) node {} -- (0.6,-1.3) [blue] node {};
   \draw (0.6,-1.3) node {} -- (1.1,0.9) [blue] node {};
   \draw (1.1,0.9) node {} -- (-1.1,0.9) [blue] node {};
   \draw (-1.1,0.9) node {} -- (-0.6,-1.3) [blue] node {};
   \draw (-0.6,-1.3) node {} -- (1.4,-0.3) [blue] node {};
   \draw (1.4,-0.3) node {} -- (0,1.5) [blue] node {};
    \put(-15,-45){$\mathbf{(17).} \hspace{0.2cm} 14$}
   \end{tikzpicture} \hspace{1cm}
\begin{tikzpicture}[thick,scale=0.4,shorten >=2pt]
    \draw (0,3) node {} -- (0,1.5) [blue] node {};
    \draw (2.3,1.9) node {} -- (1.1,0.9) [blue] node {};
    \draw (2.9,-0.6) node {} -- (1.4,-0.3) [blue] node {};
    \draw (1.3,-2.7) node {} -- (0.6,-1.3) [blue] node {};
    \draw (-2.3,1.9) node {} -- (-1.1,0.9) [red] node {};
    \draw (-2.9,-0.6) node {} -- (-1.4,-0.3) [blue] node {};
    \draw (-1.3,-2.7) node {} -- (-0.6,-1.3) [blue] node {};

   \draw (0,3) node {} -- (2.3,1.9) [blue] node {};
   \draw (2.3,1.9) node {} -- (2.9,-0.6) [red] node {};
   \draw (2.9,-0.6) node {} -- (1.3,-2.7) [blue] node {};
   \draw (1.3,-2.7) node {} -- (-1.3,-2.7) [red] node {};
   \draw (-2.3,1.9) node {} -- (-2.9,-0.6) [blue] node {};
   \draw (-2.9,-0.6) node {} -- (-1.3,-2.7) [blue] node {};
   \draw (-2.3,1.9) node {} -- (0,3) [blue] node {};

   \draw (0,1.5) node {} -- (-1.4,-0.3) [blue] node {};
   \draw (-1.4,-0.3) node {} -- (0.6,-1.3) [blue] node {};
   \draw (0.6,-1.3) node {} -- (1.1,0.9) [blue] node {};
   \draw (1.1,0.9) node {} -- (-1.1,0.9) [blue] node {};
   \draw (-1.1,0.9) node {} -- (-0.6,-1.3) [blue] node {};
   \draw (-0.6,-1.3) node {} -- (1.4,-0.3) [blue] node {};
   \draw (1.4,-0.3) node {} -- (0,1.5) [blue] node {};
   \put(-15,-45){$\mathbf{(18).} \hspace{0.2cm} 7$}
   \end{tikzpicture} \hspace{1cm}
\begin{tikzpicture}[thick,scale=0.4,shorten >=2pt]
    \draw (0,3) node {} -- (0,1.5) [blue] node {};
    \draw (2.3,1.9) node {} -- (1.1,0.9) [blue] node {};
    \draw (2.9,-0.6) node {} -- (1.4,-0.3) [blue] node {};
    \draw (1.3,-2.7) node {} -- (0.6,-1.3) [blue] node {};
    \draw (-2.3,1.9) node {} -- (-1.1,0.9) [blue] node {};
    \draw (-2.9,-0.6) node {} -- (-1.4,-0.3) [blue] node {};
    \draw (-1.3,-2.7) node {} -- (-0.6,-1.3) [blue] node {};

   \draw (0,3) node {} -- (2.3,1.9) [blue] node {};
   \draw (2.3,1.9) node {} -- (2.9,-0.6) [red] node {};%
   \draw (2.9,-0.6) node {} -- (1.3,-2.7) [blue] node {};
   \draw (1.3,-2.7) node {} -- (-1.3,-2.7) [red] node {};%
   \draw (-2.3,1.9) node {} -- (-2.9,-0.6) [blue] node {};
   \draw (-2.9,-0.6) node {} -- (-1.3,-2.7) [blue] node {};
   \draw (-2.3,1.9) node {} -- (0,3) [blue] node {};

   \draw (0,1.5) node {} -- (-1.4,-0.3) [red] node {};%
   \draw (-1.4,-0.3) node {} -- (0.6,-1.3) [blue] node {};
   \draw (0.6,-1.3) node {} -- (1.1,0.9) [blue] node {};
   \draw (1.1,0.9) node {} -- (-1.1,0.9) [blue] node {};
   \draw (-1.1,0.9) node {} -- (-0.6,-1.3) [blue] node {};
   \draw (-0.6,-1.3) node {} -- (1.4,-0.3) [blue] node {};
   \draw (1.4,-0.3) node {} -- (0,1.5) [blue] node {};
    \put(-15,-45){$\mathbf{(19).} \hspace{0.2cm} 7$}
   \end{tikzpicture} \hspace{1cm}
\begin{tikzpicture}[thick,scale=0.4,shorten >=2pt]
    \draw (0,3) node {} -- (0,1.5) [blue] node {};
    \draw (2.3,1.9) node {} -- (1.1,0.9) [blue] node {};
    \draw (2.9,-0.6) node {} -- (1.4,-0.3) [blue] node {};
    \draw (1.3,-2.7) node {} -- (0.6,-1.3) [blue] node {};
    \draw (-2.3,1.9) node {} -- (-1.1,0.9) [blue] node {};
    \draw (-2.9,-0.6) node {} -- (-1.4,-0.3) [blue] node {};
    \draw (-1.3,-2.7) node {} -- (-0.6,-1.3) [blue] node {};

   \draw (0,3) node {} -- (2.3,1.9) [blue] node {};
   \draw (2.3,1.9) node {} -- (2.9,-0.6) [red] node {};%
   \draw (2.9,-0.6) node {} -- (1.3,-2.7) [blue] node {};
   \draw (1.3,-2.7) node {} -- (-1.3,-2.7) [red] node {};%
   \draw (-2.3,1.9) node {} -- (-2.9,-0.6) [blue] node {};
   \draw (-2.9,-0.6) node {} -- (-1.3,-2.7) [blue] node {};
   \draw (-2.3,1.9) node {} -- (0,3) [blue] node {};

   \draw (0,1.5) node {} -- (-1.4,-0.3) [blue] node {};
   \draw (-1.4,-0.3) node {} -- (0.6,-1.3) [blue] node {};
   \draw (0.6,-1.3) node {} -- (1.1,0.9) [blue] node {};
   \draw (1.1,0.9) node {} -- (-1.1,0.9) [blue] node {};
   \draw (-1.1,0.9) node {} -- (-0.6,-1.3) [red] node {};%
   \draw (-0.6,-1.3) node {} -- (1.4,-0.3) [blue] node {};
   \draw (1.4,-0.3) node {} -- (0,1.5) [blue] node {};
    \put(-15,-45){$\mathbf{(20).} \hspace{0.2cm} 7$}
   \end{tikzpicture} \\ \vspace{1cm}
\begin{tikzpicture}[thick,scale=0.4,shorten >=2pt]
    \draw (0,3) node {} -- (0,1.5) [blue] node {};
    \draw (2.3,1.9) node {} -- (1.1,0.9) [blue] node {};
    \draw (2.9,-0.6) node {} -- (1.4,-0.3) [blue] node {};
    \draw (1.3,-2.7) node {} -- (0.6,-1.3) [blue] node {};
    \draw (-2.3,1.9) node {} -- (-1.1,0.9) [blue] node {};
    \draw (-2.9,-0.6) node {} -- (-1.4,-0.3) [blue] node {};
    \draw (-1.3,-2.7) node {} -- (-0.6,-1.3) [blue] node {};

   \draw (0,3) node {} -- (2.3,1.9) [blue] node {};
   \draw (2.3,1.9) node {} -- (2.9,-0.6) [blue] node {};
   \draw (2.9,-0.6) node {} -- (1.3,-2.7) [blue] node {};
   \draw (1.3,-2.7) node {} -- (-1.3,-2.7) [blue] node {};
   \draw (-2.3,1.9) node {} -- (-2.9,-0.6) [blue] node {};
   \draw (-2.9,-0.6) node {} -- (-1.3,-2.7) [blue] node {};
   \draw (-2.3,1.9) node {} -- (0,3) [blue] node {};

   \draw (0,1.5) node {} -- (-1.4,-0.3) [blue] node {};
   \draw (-1.4,-0.3) node {} -- (0.6,-1.3) [red] node {};%
   \draw (0.6,-1.3) node {} -- (1.1,0.9) [blue] node {};
   \draw (1.1,0.9) node {} -- (-1.1,0.9) [red] node {};%
   \draw (-1.1,0.9) node {} -- (-0.6,-1.3) [blue] node {};
   \draw (-0.6,-1.3) node {} -- (1.4,-0.3) [red] node {};%
   \draw (1.4,-0.3) node {} -- (0,1.5) [blue] node {};
    \put(-15,-45){$\mathbf{(21).} \hspace{0.2cm} 7$}
   \end{tikzpicture} \hspace{1cm}
\begin{tikzpicture}[thick,scale=0.4,shorten >=2pt]
    \draw (0,3) node {} -- (0,1.5) [blue] node {};
    \draw (2.3,1.9) node {} -- (1.1,0.9) [blue] node {};
    \draw (2.9,-0.6) node {} -- (1.4,-0.3) [blue] node {};
    \draw (1.3,-2.7) node {} -- (0.6,-1.3) [blue] node {};
    \draw (-2.3,1.9) node {} -- (-1.1,0.9) [blue] node {};
    \draw (-2.9,-0.6) node {} -- (-1.4,-0.3) [blue] node {};
    \draw (-1.3,-2.7) node {} -- (-0.6,-1.3) [blue] node {};

   \draw (0,3) node {} -- (2.3,1.9) [red] node {};%
   \draw (2.3,1.9) node {} -- (2.9,-0.6) [blue] node {};
   \draw (2.9,-0.6) node {} -- (1.3,-2.7) [red] node {};%
   \draw (1.3,-2.7) node {} -- (-1.3,-2.7) [blue] node {};
   \draw (-2.3,1.9) node {} -- (-2.9,-0.6) [blue] node {};
   \draw (-2.9,-0.6) node {} -- (-1.3,-2.7) [red] node {};%
   \draw (-2.3,1.9) node {} -- (0,3) [blue] node {};

   \draw (0,1.5) node {} -- (-1.4,-0.3) [red] node {};%
   \draw (-1.4,-0.3) node {} -- (0.6,-1.3) [blue] node {};
   \draw (0.6,-1.3) node {} -- (1.1,0.9) [blue] node {};
   \draw (1.1,0.9) node {} -- (-1.1,0.9) [blue] node {};
   \draw (-1.1,0.9) node {} -- (-0.6,-1.3) [blue] node {};
   \draw (-0.6,-1.3) node {} -- (1.4,-0.3) [blue] node {};
   \draw (1.4,-0.3) node {} -- (0,1.5) [blue] node {};
    \put(-15,-45){$\mathbf{(22).} \hspace{0.2cm} 1$}
   \end{tikzpicture} \hspace{1cm}
\begin{tikzpicture}[thick,scale=0.4,shorten >=2pt]
    \draw (0,3) node {} -- (0,1.5) [red] node {};%
    \draw (2.3,1.9) node {} -- (1.1,0.9) [blue] node {};
    \draw (2.9,-0.6) node {} -- (1.4,-0.3) [blue] node {};
    \draw (1.3,-2.7) node {} -- (0.6,-1.3) [blue] node {};
    \draw (-2.3,1.9) node {} -- (-1.1,0.9)  [blue] node {};
    \draw (-2.9,-0.6) node {} -- (-1.4,-0.3) [red] node {};%
    \draw (-1.3,-2.7) node {} -- (-0.6,-1.3) [blue] node {};

   \draw (0,3) node {} -- (2.3,1.9) [blue] node {};
   \draw (2.3,1.9) node {} -- (2.9,-0.6) [red] node {};%
   \draw (2.9,-0.6) node {} -- (1.3,-2.7) [blue] node {};
   \draw (1.3,-2.7) node {} -- (-1.3,-2.7) [blue] node {};
   \draw (-2.3,1.9) node {} -- (-2.9,-0.6) [blue] node {};
   \draw (-2.9,-0.6) node {} -- (-1.3,-2.7) [blue] node {};
   \draw (-2.3,1.9) node {} -- (0,3) [blue] node {};

   \draw (0,1.5) node {} -- (-1.4,-0.3) [blue] node {};
   \draw (-1.4,-0.3) node {} -- (0.6,-1.3) [blue] node {};
   \draw (0.6,-1.3) node {} -- (1.1,0.9) [blue] node {};
   \draw (1.1,0.9) node {} -- (-1.1,0.9) [blue] node {};
   \draw (-1.1,0.9) node {} -- (-0.6,-1.3) [blue] node {};
   \draw (-0.6,-1.3) node {} -- (1.4,-0.3) [blue] node {};
   \draw (1.4,-0.3) node {} -- (0,1.5) [blue] node {};
    \put(-15,-45){$\mathbf{(23).} \hspace{0.2cm} 7$}
   \end{tikzpicture} \hspace{1cm}
\begin{tikzpicture}[thick,scale=0.4,shorten >=2pt]
    \draw (0,3) node {} -- (0,1.5) [blue] node {};
    \draw (2.3,1.9) node {} -- (1.1,0.9) [blue] node {};
    \draw (2.9,-0.6) node {} -- (1.4,-0.3) [blue] node {};
    \draw (1.3,-2.7) node {} -- (0.6,-1.3) [red] node {}; %
    \draw (-2.3,1.9) node {} -- (-1.1,0.9) [blue] node {};
    \draw (-2.9,-0.6) node {} -- (-1.4,-0.3) [red] node {};%
    \draw (-1.3,-2.7) node {} -- (-0.6,-1.3) [blue] node {};

   \draw (0,3) node {} -- (2.3,1.9) [red] node {};%
   \draw (2.3,1.9) node {} -- (2.9,-0.6) [blue] node {};
   \draw (2.9,-0.6) node {} -- (1.3,-2.7) [blue] node {};
   \draw (1.3,-2.7) node {} -- (-1.3,-2.7) [blue] node {};
   \draw (-2.3,1.9) node {} -- (-2.9,-0.6) [blue] node {};
   \draw (-2.9,-0.6) node {} -- (-1.3,-2.7) [blue] node {};
   \draw (-2.3,1.9) node {} -- (0,3) [blue] node {};

   \draw (0,1.5) node {} -- (-1.4,-0.3) [blue] node {};
   \draw (-1.4,-0.3) node {} -- (0.6,-1.3) [blue] node {};
   \draw (0.6,-1.3) node {} -- (1.1,0.9) [blue] node {};
   \draw (1.1,0.9) node {} -- (-1.1,0.9) [blue] node {};
   \draw (-1.1,0.9) node {} -- (-0.6,-1.3) [blue] node {};
   \draw (-0.6,-1.3) node {} -- (1.4,-0.3) [blue] node {};
   \draw (1.4,-0.3) node {} -- (0,1.5) [blue] node {};
   \put(-15,-45){$\mathbf{(24).} \hspace{0.2cm} 7$}
   \end{tikzpicture} \\ \vspace{1cm}
\begin{tikzpicture}[thick,scale=0.4,shorten >=2pt]
    \draw (0,3) node {} -- (0,1.5) [blue] node {};
    \draw (2.3,1.9) node {} -- (1.1,0.9) [red] node {};%
    \draw (2.9,-0.6) node {} -- (1.4,-0.3) [blue] node {};
    \draw (1.3,-2.7) node {} -- (0.6,-1.3) [blue] node {};
    \draw (-2.3,1.9) node {} -- (-1.1,0.9) [blue] node {};
    \draw (-2.9,-0.6) node {} -- (-1.4,-0.3) [blue] node {};
    \draw (-1.3,-2.7) node {} -- (-0.6,-1.3) [blue] node {};

   \draw (0,3) node {} -- (2.3,1.9) [blue] node {};
   \draw (2.3,1.9) node {} -- (2.9,-0.6) [blue] node {};
   \draw (2.9,-0.6) node {} -- (1.3,-2.7) [blue] node {};
   \draw (1.3,-2.7) node {} -- (-1.3,-2.7) [red] node {};%
   \draw (-2.3,1.9) node {} -- (-2.9,-0.6) [blue] node {};
   \draw (-2.9,-0.6) node {} -- (-1.3,-2.7) [blue] node {};
   \draw (-2.3,1.9) node {} -- (0,3) [blue] node {};

   \draw (0,1.5) node {} -- (-1.4,-0.3) [blue] node {};
   \draw (-1.4,-0.3) node {} -- (0.6,-1.3) [blue] node {};
   \draw (0.6,-1.3) node {} -- (1.1,0.9) [blue] node {};
   \draw (1.1,0.9) node {} -- (-1.1,0.9) [blue] node {};
   \draw (-1.1,0.9) node {} -- (-0.6,-1.3) [blue] node {};
   \draw (-0.6,-1.3) node {} -- (1.4,-0.3) [blue] node {};
   \draw (1.4,-0.3) node {} -- (0,1.5) [red] node {};%
   \put(-15,-45){$\mathbf{(25).} \hspace{0.2cm} 7$}
   \end{tikzpicture} \hspace{1cm}
\begin{tikzpicture}[thick,scale=0.4,shorten >=2pt]
    \draw (0,3) node {} -- (0,1.5) [blue] node {};
    \draw (2.3,1.9) node {} -- (1.1,0.9) [red] node {};%
    \draw (2.9,-0.6) node {} -- (1.4,-0.3) [blue] node {};
    \draw (1.3,-2.7) node {} -- (0.6,-1.3) [blue] node {};
    \draw (-2.3,1.9) node {} -- (-1.1,0.9) [blue] node {};
    \draw (-2.9,-0.6) node {} -- (-1.4,-0.3) [blue] node {};
    \draw (-1.3,-2.7) node {} -- (-0.6,-1.3) [blue] node {};

   \draw (0,3) node {} -- (2.3,1.9) [blue] node {};
   \draw (2.3,1.9) node {} -- (2.9,-0.6) [blue] node {};
   \draw (2.9,-0.6) node {} -- (1.3,-2.7) [red] node {};%
   \draw (1.3,-2.7) node {} -- (-1.3,-2.7) [blue] node {};
   \draw (-2.3,1.9) node {} -- (-2.9,-0.6) [blue] node {};
   \draw (-2.9,-0.6) node {} -- (-1.3,-2.7) [blue] node {};
   \draw (-2.3,1.9) node {} -- (0,3) [blue] node {};

   \draw (0,1.5) node {} -- (-1.4,-0.3) [blue] node {};
   \draw (-1.4,-0.3) node {} -- (0.6,-1.3) [blue] node {};
   \draw (0.6,-1.3) node {} -- (1.1,0.9) [blue] node {};
   \draw (1.1,0.9) node {} -- (-1.1,0.9) [blue] node {};
   \draw (-1.1,0.9) node {} -- (-0.6,-1.3) [red] node {};%
   \draw (-0.6,-1.3) node {} -- (1.4,-0.3) [blue] node {};
   \draw (1.4,-0.3) node {} -- (0,1.5) [blue] node {};
   \put(-15,-45){$\mathbf{(26).} \hspace{0.2cm} 7$}
   \end{tikzpicture} \hspace{1cm}
\begin{tikzpicture}[thick,scale=0.4,shorten >=2pt]
    \draw (0,3) node {} -- (0,1.5) [blue] node {};
    \draw (2.3,1.9) node {} -- (1.1,0.9) [red] node {};%
    \draw (2.9,-0.6) node {} -- (1.4,-0.3) [blue] node {};
    \draw (1.3,-2.7) node {} -- (0.6,-1.3) [blue] node {};
    \draw (-2.3,1.9) node {} -- (-1.1,0.9) [blue] node {};
    \draw (-2.9,-0.6) node {} -- (-1.4,-0.3) [blue] node {};
    \draw (-1.3,-2.7) node {} -- (-0.6,-1.3) [blue] node {};

   \draw (0,3) node {} -- (2.3,1.9) [blue] node {};
   \draw (2.3,1.9) node {} -- (2.9,-0.6) [blue] node {};
   \draw (2.9,-0.6) node {} -- (1.3,-2.7) [blue] node {};
   \draw (1.3,-2.7) node {} -- (-1.3,-2.7) [red] node {};%
   \draw (-2.3,1.9) node {} -- (-2.9,-0.6) [blue] node {};
   \draw (-2.9,-0.6) node {} -- (-1.3,-2.7) [blue] node {};
   \draw (-2.3,1.9) node {} -- (0,3) [blue] node {};

   \draw (0,1.5) node {} -- (-1.4,-0.3) [blue] node {};
   \draw (-1.4,-0.3) node {} -- (0.6,-1.3) [blue] node {};
   \draw (0.6,-1.3) node {} -- (1.1,0.9) [blue] node {};
   \draw (1.1,0.9) node {} -- (-1.1,0.9) [blue] node {};
   \draw (-1.1,0.9) node {} -- (-0.6,-1.3) [red] node {};%
   \draw (-0.6,-1.3) node {} -- (1.4,-0.3) [blue] node {};
   \draw (1.4,-0.3) node {} -- (0,1.5) [blue] node {};
   \put(-15,-45){$\mathbf{(27).} \hspace{0.2cm} 7$}
      \end{tikzpicture} \hspace{1cm}
\begin{tikzpicture}[thick,scale=0.4,shorten >=2pt]
    \draw (0,3) node {} -- (0,1.5) [blue] node {};
    \draw (2.3,1.9) node {} -- (1.1,0.9) [red] node {};%
    \draw (2.9,-0.6) node {} -- (1.4,-0.3) [blue] node {};
    \draw (1.3,-2.7) node {} -- (0.6,-1.3) [blue] node {};
    \draw (-2.3,1.9) node {} -- (-1.1,0.9) [blue] node {};
    \draw (-2.9,-0.6) node {} -- (-1.4,-0.3) [blue] node {};
    \draw (-1.3,-2.7) node {} -- (-0.6,-1.3) [blue] node {};

   \draw (0,3) node {} -- (2.3,1.9) [blue] node {};
   \draw (2.3,1.9) node {} -- (2.9,-0.6) [blue] node {};
   \draw (2.9,-0.6) node {} -- (1.3,-2.7) [blue] node {};
   \draw (1.3,-2.7) node {} -- (-1.3,-2.7) [red] node {};%
   \draw (-2.3,1.9) node {} -- (-2.9,-0.6) [blue] node {};
   \draw (-2.9,-0.6) node {} -- (-1.3,-2.7) [blue] node {};
   \draw (-2.3,1.9) node {} -- (0,3) [blue] node {};

   \draw (0,1.5) node {} -- (-1.4,-0.3) [red] node {};%
   \draw (-1.4,-0.3) node {} -- (0.6,-1.3) [blue] node {};
   \draw (0.6,-1.3) node {} -- (1.1,0.9) [blue] node {};
   \draw (1.1,0.9) node {} -- (-1.1,0.9) [blue] node {};
   \draw (-1.1,0.9) node {} -- (-0.6,-1.3) [blue] node {};
   \draw (-0.6,-1.3) node {} -- (1.4,-0.3) [blue] node {};
   \draw (1.4,-0.3) node {} -- (0,1.5) [blue] node {};
   \put(-15,-45){$\mathbf{(28).} \hspace{0.2cm} 7$}
      \end{tikzpicture} \\ \vspace{1cm}
\begin{tikzpicture}[thick,scale=0.4,shorten >=2pt]
    \draw (0,3) node {} -- (0,1.5) [blue] node {};
    \draw (2.3,1.9) node {} -- (1.1,0.9) [blue] node {};
    \draw (2.9,-0.6) node {} -- (1.4,-0.3) [blue] node {};
    \draw (1.3,-2.7) node {} -- (0.6,-1.3) [blue] node {};
    \draw (-2.3,1.9) node {} -- (-1.1,0.9)  [blue] node {};
    \draw (-2.9,-0.6) node {} -- (-1.4,-0.3) [red] node {};%
    \draw (-1.3,-2.7) node {} -- (-0.6,-1.3) [blue] node {};

   \draw (0,3) node {} -- (2.3,1.9) [red] node {};%
   \draw (2.3,1.9) node {} -- (2.9,-0.6) [blue] node {};
   \draw (2.9,-0.6) node {} -- (1.3,-2.7) [red] node {};%
   \draw (1.3,-2.7) node {} -- (-1.3,-2.7) [blue] node {};
   \draw (-2.3,1.9) node {} -- (-2.9,-0.6) [blue] node {};
   \draw (-2.9,-0.6) node {} -- (-1.3,-2.7) [blue] node {};
   \draw (-2.3,1.9) node {} -- (0,3) [blue] node {};

   \draw (0,1.5) node {} -- (-1.4,-0.3) [blue] node {};
   \draw (-1.4,-0.3) node {} -- (0.6,-1.3) [blue] node {};
   \draw (0.6,-1.3) node {} -- (1.1,0.9) [blue] node {};
   \draw (1.1,0.9) node {} -- (-1.1,0.9) [blue] node {};
   \draw (-1.1,0.9) node {} -- (-0.6,-1.3) [red] node {};%
   \draw (-0.6,-1.3) node {} -- (1.4,-0.3) [blue] node {};
   \draw (1.4,-0.3) node {} -- (0,1.5) [blue] node {};
   \put(-15,-45){$\mathbf{(29).} \hspace{0.2cm} 1$}
   \end{tikzpicture} \hspace{1cm}
\begin{tikzpicture}[thick,scale=0.4,shorten >=2pt]
    \draw (0,3) node {} -- (0,1.5) [red] node {};
    \draw (2.3,1.9) node {} -- (1.1,0.9) [blue] node {};
    \draw (2.9,-0.6) node {} -- (1.4,-0.3) [blue] node {};
    \draw (1.3,-2.7) node {} -- (0.6,-1.3) [blue] node {};
    \draw (-2.3,1.9) node {} -- (-1.1,0.9) [blue] node {};
    \draw (-2.9,-0.6) node {} -- (-1.4,-0.3) [blue] node {};
    \draw (-1.3,-2.7) node {} -- (-0.6,-1.3) [blue] node {};

   \draw (0,3) node {} -- (2.3,1.9) [blue] node {};
   \draw (2.3,1.9) node {} -- (2.9,-0.6) [red] node {};%
   \draw (2.9,-0.6) node {} -- (1.3,-2.7) [blue] node {};
   \draw (1.3,-2.7) node {} -- (-1.3,-2.7) [blue] node {};
   \draw (-2.3,1.9) node {} -- (-2.9,-0.6) [red] node {};%
   \draw (-2.9,-0.6) node {} -- (-1.3,-2.7) [blue] node {};
   \draw (-2.3,1.9) node {} -- (0,3) [blue] node {};

   \draw (0,1.5) node {} -- (-1.4,-0.3) [blue] node {};
   \draw (-1.4,-0.3) node {} -- (0.6,-1.3) [red] node {};%
   \draw (0.6,-1.3) node {} -- (1.1,0.9) [blue] node {};
   \draw (1.1,0.9) node {} -- (-1.1,0.9) [blue] node {};
   \draw (-1.1,0.9) node {} -- (-0.6,-1.3) [blue] node {};
   \draw (-0.6,-1.3) node {} -- (1.4,-0.3) [blue] node {};
   \draw (1.4,-0.3) node {} -- (0,1.5) [blue] node {};
   \put(-15,-45){$\mathbf{(30).} \hspace{0.2cm} 7$}
   \end{tikzpicture} \hspace{1cm}
\begin{tikzpicture}[thick,scale=0.4,shorten >=2pt]
    \draw (0,3) node {} -- (0,1.5) [blue] node {};
    \draw (2.3,1.9) node {} -- (1.1,0.9) [blue] node {};
    \draw (2.9,-0.6) node {} -- (1.4,-0.3) [blue] node {};
    \draw (1.3,-2.7) node {} -- (0.6,-1.3) [blue] node {};
    \draw (-2.3,1.9) node {} -- (-1.1,0.9) [blue] node {};
    \draw (-2.9,-0.6) node {} -- (-1.4,-0.3) [blue] node {};
    \draw (-1.3,-2.7) node {} -- (-0.6,-1.3) [blue] node {};

   \draw (0,3) node {} -- (2.3,1.9) [blue] node {};
   \draw (2.3,1.9) node {} -- (2.9,-0.6) [blue] node {};
   \draw (2.9,-0.6) node {} -- (1.3,-2.7) [blue] node {};
   \draw (1.3,-2.7) node {} -- (-1.3,-2.7) [blue] node {};
   \draw (-2.3,1.9) node {} -- (-2.9,-0.6) [blue] node {};
   \draw (-2.9,-0.6) node {} -- (-1.3,-2.7) [blue] node {};
   \draw (-2.3,1.9) node {} -- (0,3) [red] node {};%

   \draw (0,1.5) node {} -- (-1.4,-0.3) [blue] node {};
   \draw (-1.4,-0.3) node {} -- (0.6,-1.3) [blue] node {};
   \draw (0.6,-1.3) node {} -- (1.1,0.9) [blue] node {};
   \draw (1.1,0.9) node {} -- (-1.1,0.9) [blue] node {};
   \draw (-1.1,0.9) node {} -- (-0.6,-1.3) [blue] node {};
   \draw (-0.6,-1.3) node {} -- (1.4,-0.3) [red] node {};%
   \draw (1.4,-0.3) node {} -- (0,1.5) [blue] node {};
     \put(-15,-45){$\mathbf{(31).} \hspace{0.2cm} 7$}
   \end{tikzpicture} \hspace{1cm}
\begin{tikzpicture}[thick,scale=0.4,shorten >=2pt]
    \draw (0,3) node {} -- (0,1.5) [blue] node {};
    \draw (2.3,1.9) node {} -- (1.1,0.9) [blue] node {};
    \draw (2.9,-0.6) node {} -- (1.4,-0.3) [blue] node {};
    \draw (1.3,-2.7) node {} -- (0.6,-1.3) [blue] node {};
    \draw (-2.3,1.9) node {} -- (-1.1,0.9) [blue] node {};
    \draw (-2.9,-0.6) node {} -- (-1.4,-0.3) [blue] node {};
    \draw (-1.3,-2.7) node {} -- (-0.6,-1.3) [blue] node {};

   \draw (0,3) node {} -- (2.3,1.9) [blue] node {};
   \draw (2.3,1.9) node {} -- (2.9,-0.6) [red] node {};%
   \draw (2.9,-0.6) node {} -- (1.3,-2.7) [blue] node {};
   \draw (1.3,-2.7) node {} -- (-1.3,-2.7) [blue] node {};
   \draw (-2.3,1.9) node {} -- (-2.9,-0.6) [blue] node {};
   \draw (-2.9,-0.6) node {} -- (-1.3,-2.7) [red] node {};%
   \draw (-2.3,1.9) node {} -- (0,3) [blue] node {};

   \draw (0,1.5) node {} -- (-1.4,-0.3) [red] node {};%
   \draw (-1.4,-0.3) node {} -- (0.6,-1.3) [blue] node {};
   \draw (0.6,-1.3) node {} -- (1.1,0.9) [blue] node {};
   \draw (1.1,0.9) node {} -- (-1.1,0.9) [blue] node {};
   \draw (-1.1,0.9) node {} -- (-0.6,-1.3) [blue] node {};
   \draw (-0.6,-1.3) node {} -- (1.4,-0.3) [blue] node {};
   \draw (1.4,-0.3) node {} -- (0,1.5) [blue] node {};
     \put(-15,-45){$\mathbf{(32).} \hspace{0.2cm} 7$}
   \end{tikzpicture} \\ \vspace{1cm}
\begin{tikzpicture}[thick,scale=0.4,shorten >=2pt]
    \draw (0,3) node {} -- (0,1.5) [blue] node {};
    \draw (2.3,1.9) node {} -- (1.1,0.9) [blue] node {};
    \draw (2.9,-0.6) node {} -- (1.4,-0.3) [blue] node {};
    \draw (1.3,-2.7) node {} -- (0.6,-1.3) [blue] node {};
    \draw (-2.3,1.9) node {} -- (-1.1,0.9) [blue] node {};
    \draw (-2.9,-0.6) node {} -- (-1.4,-0.3) [blue] node {};
    \draw (-1.3,-2.7) node {} -- (-0.6,-1.3) [blue] node {};

   \draw (0,3) node {} -- (2.3,1.9) [blue] node {};
   \draw (2.3,1.9) node {} -- (2.9,-0.6) [red] node {};%
   \draw (2.9,-0.6) node {} -- (1.3,-2.7) [blue] node {};
   \draw (1.3,-2.7) node {} -- (-1.3,-2.7) [blue] node {};
   \draw (-2.3,1.9) node {} -- (-2.9,-0.6) [blue] node {};
   \draw (-2.9,-0.6) node {} -- (-1.3,-2.7) [blue] node {};
   \draw (-2.3,1.9) node {} -- (0,3) [blue] node {};

   \draw (0,1.5) node {} -- (-1.4,-0.3) [red] node {};%
   \draw (-1.4,-0.3) node {} -- (0.6,-1.3) [blue] node {};
   \draw (0.6,-1.3) node {} -- (1.1,0.9) [blue] node {};
   \draw (1.1,0.9) node {} -- (-1.1,0.9) [blue] node {};
   \draw (-1.1,0.9) node {} -- (-0.6,-1.3) [red] node {};%
   \draw (-0.6,-1.3) node {} -- (1.4,-0.3) [blue] node {};
   \draw (1.4,-0.3) node {} -- (0,1.5) [blue] node {};
     \put(-15,-45){$\mathbf{(33).} \hspace{0.2cm} 7$}
   \end{tikzpicture} \hspace{1cm}
\begin{tikzpicture}[thick,scale=0.4,shorten >=2pt]
    \draw (0,3) node {} -- (0,1.5) [blue] node {};
    \draw (2.3,1.9) node {} -- (1.1,0.9) [blue] node {};
    \draw (2.9,-0.6) node {} -- (1.4,-0.3) [blue] node {};
    \draw (1.3,-2.7) node {} -- (0.6,-1.3) [blue] node {};
    \draw (-2.3,1.9) node {} -- (-1.1,0.9) [blue] node {};
    \draw (-2.9,-0.6) node {} -- (-1.4,-0.3) [blue] node {};
    \draw (-1.3,-2.7) node {} -- (-0.6,-1.3) [blue] node {};

   \draw (0,3) node {} -- (2.3,1.9) [red] node {};%
   \draw (2.3,1.9) node {} -- (2.9,-0.6) [blue] node {};
   \draw (2.9,-0.6) node {} -- (1.3,-2.7) [blue] node {};
   \draw (1.3,-2.7) node {} -- (-1.3,-2.7) [blue] node {};
   \draw (-2.3,1.9) node {} -- (-2.9,-0.6) [blue] node {};
   \draw (-2.9,-0.6) node {} -- (-1.3,-2.7) [blue] node {};
   \draw (-2.3,1.9) node {} -- (0,3) [blue] node {};

   \draw (0,1.5) node {} -- (-1.4,-0.3) [blue] node {};
   \draw (-1.4,-0.3) node {} -- (0.6,-1.3) [blue] node {};
   \draw (0.6,-1.3) node {} -- (1.1,0.9) [blue] node {};
   \draw (1.1,0.9) node {} -- (-1.1,0.9) [blue] node {};
   \draw (-1.1,0.9) node {} -- (-0.6,-1.3) [red] node {};%
   \draw (-0.6,-1.3) node {} -- (1.4,-0.3) [blue] node {};
   \draw (1.4,-0.3) node {} -- (0,1.5) [blue] node {};
     \put(-15,-45){$\mathbf{(34).} \hspace{0.2cm} 14$}
   \end{tikzpicture} \hspace{1cm}
\begin{tikzpicture}[thick,scale=0.4,shorten >=2pt]
    \draw (0,3) node {} -- (0,1.5) [blue] node {};
    \draw (2.3,1.9) node {} -- (1.1,0.9) [blue] node {};
    \draw (2.9,-0.6) node {} -- (1.4,-0.3) [blue] node {};
    \draw (1.3,-2.7) node {} -- (0.6,-1.3) [blue] node {};
    \draw (-2.3,1.9) node {} -- (-1.1,0.9) [blue] node {};
    \draw (-2.9,-0.6) node {} -- (-1.4,-0.3) [blue] node {};
    \draw (-1.3,-2.7) node {} -- (-0.6,-1.3) [blue] node {};

   \draw (0,3) node {} -- (2.3,1.9) [blue] node {};
   \draw (2.3,1.9) node {} -- (2.9,-0.6) [red] node {};%
   \draw (2.9,-0.6) node {} -- (1.3,-2.7) [blue] node {};
   \draw (1.3,-2.7) node {} -- (-1.3,-2.7) [blue] node {};
   \draw (-2.3,1.9) node {} -- (-2.9,-0.6) [blue] node {};
   \draw (-2.9,-0.6) node {} -- (-1.3,-2.7) [blue] node {};
   \draw (-2.3,1.9) node {} -- (0,3) [blue] node {};

   \draw (0,1.5) node {} -- (-1.4,-0.3) [blue] node {};
   \draw (-1.4,-0.3) node {} -- (0.6,-1.3) [blue] node {};
   \draw (0.6,-1.3) node {} -- (1.1,0.9) [red] node {};%
   \draw (1.1,0.9) node {} -- (-1.1,0.9) [blue] node {};
   \draw (-1.1,0.9) node {} -- (-0.6,-1.3) [blue] node {};
   \draw (-0.6,-1.3) node {} -- (1.4,-0.3) [blue] node {};
   \draw (1.4,-0.3) node {} -- (0,1.5) [red] node {};%
     \put(-15,-45){$\mathbf{(35).} \hspace{0.2cm} 1$}
   \end{tikzpicture} \hspace{1cm}
\begin{tikzpicture}[thick,scale=0.4,shorten >=2pt]
    \draw (0,3) node {} -- (0,1.5) [red] node {};%
    \draw (2.3,1.9) node {} -- (1.1,0.9) [blue] node {};
    \draw (2.9,-0.6) node {} -- (1.4,-0.3) [blue] node {};
    \draw (1.3,-2.7) node {} -- (0.6,-1.3) [blue] node {};
    \draw (-2.3,1.9) node {} -- (-1.1,0.9) [blue] node {};
    \draw (-2.9,-0.6) node {} -- (-1.4,-0.3) [blue] node {};
    \draw (-1.3,-2.7) node {} -- (-0.6,-1.3) [blue] node {};

   \draw (0,3) node {} -- (2.3,1.9) [blue] node {};
   \draw (2.3,1.9) node {} -- (2.9,-0.6) [red] node {};%
   \draw (2.9,-0.6) node {} -- (1.3,-2.7) [blue] node {};
   \draw (1.3,-2.7) node {} -- (-1.3,-2.7) [blue] node {};
   \draw (-2.3,1.9) node {} -- (-2.9,-0.6) [blue] node {};
   \draw (-2.9,-0.6) node {} -- (-1.3,-2.7) [blue] node {};
   \draw (-2.3,1.9) node {} -- (0,3) [blue] node {};

   \draw (0,1.5) node {} -- (-1.4,-0.3) [blue] node {};
   \draw (-1.4,-0.3) node {} -- (0.6,-1.3) [blue] node {};
   \draw (0.6,-1.3) node {} -- (1.1,0.9) [red] node {};%
   \draw (1.1,0.9) node {} -- (-1.1,0.9) [blue] node {};
   \draw (-1.1,0.9) node {} -- (-0.6,-1.3) [blue] node {};
   \draw (-0.6,-1.3) node {} -- (1.4,-0.3) [blue] node {};
   \draw (1.4,-0.3) node {} -- (0,1.5) [blue] node {};
     \put(-15,-45){$\mathbf{(36).} \hspace{0.2cm} 7$}
   \end{tikzpicture}
\end{scriptsize}
\end{center}

\noindent
 {\sc Yousef Bagheri, Ali Reza Moghaddamfar and Farzaneh Ramezani }\\[0.2cm]
{\sc Faculty of Mathematics, K. N. Toosi
University of Technology,
 P. O. Box $16765$--$3381$, Tehran, Iran,}\\[0.1cm]
{\em E-mail addresses}: {\tt moghadam@kntu.ac.ir},  {\tt ramezani@kntu.ac.ir}
\end{document}